\documentclass[12pt]{article}
\usepackage{amssymb}
\usepackage{mathrsfs}
\usepackage[hypertex]{hyperref}
\usepackage{amsfonts}
\usepackage{graphicx}
\usepackage{mathptmx}
\usepackage{latexsym,amsmath,amssymb,amsfonts,amsthm}

\newtheorem{theorem}{Theorem}[section]
\newtheorem{lemma}[theorem]{Lemma}

\newtheorem{remark}[theorem]{Remark}

\begin{document}
\setcounter{page}{1}
\title{An isoperimetric inequality for a biharmonic Steklov problem}
\author{Shan Li,~ Jing Mao$^{\ast}$}

\date{}
\protect\footnotetext{\!\!\!\!\!\!\!\!\!\!\!\!{$^{\ast}$Corresponding author}\\
{MSC 2020: 35P15, 53C42.}
\\
{ ~~Key Words: Eigenvalues; The biharmonic operator; Steklov-type
eigenvalue problem; Poisson's ratio.}}
\maketitle ~~~\\[-15mm]

\begin{center}
{\footnotesize Faculty of Mathematics and Statistics, \\
Key Laboratory of Applied Mathematics of Hubei Province, \\
Hubei University, Wuhan 430062, China \\
Email: jiner120@163.com }
\end{center}


\begin{abstract}
For the biharmonic Steklov eigenvalue problem considered in this
paper, we show that among all bounded Euclidean domains of class
$C^{1}$ with fixed measure, the ball maximizes the first positive
eigenvalue.
 \end{abstract}

\markright{\sl\hfill S. Li, J. Mao \hfill}

\section{Introduction}  \label{intro}
\renewcommand{\thesection}{\arabic{section}}
\renewcommand{\theequation}{\thesection.\arabic{equation}}
\setcounter{equation}{0}

The study of isoperimetric problems is a hot topic in Differential
Geometry and of course in Spectral Geometry. The famous Faber-Krahn
inequality says that the ball minimizes the first Dirichlet
eigenvalue of the Laplacian among all domains with fixed measure
(see \cite{fa,kr}). While Szeg\H{o} \cite{gs1,gs2} and Weinberg
\cite{hfw} showed that among all domains with fixed measure, the
ball maximizes the first nonzero Neumann eigenvalue of the
Laplacian. Except Dirichlet and Neumann cases, other boundary
conditions for this kind of isoperimetric problems can also be
proposed and similar conclusions can be expected (see, e.g.,
\cite{bro}). These conclusions reveal the dependence of eigenvalues
of the Laplacian on bounded Euclidean domains with fixed volume.
Naturally, one might ask:

$\\$\textbf{Question}. \emph{Can similar spectral isoperimetric
inequalities be obtained for other elliptic operators?}

  $\\$The
answer is of course positive and many facts have been known. In
fact, even for \emph{the biharmonic operator} (also called \emph{the
bi-Laplace operator}), although generally the corresponding
eigenvalue equations (with different boundary conditions) are
\emph{fourth-order} PDEs, some spectral isoperimetric inequalities
can also be achieved. For instance,

\begin{itemize}

\item Lord Rayleigh conjectured:

\begin{itemize}

\item \emph{Among open sets in the Euclidean $n$-space $\mathbb{R}^n$
with the same measure, the ball minimizes the fundamental tone of
the clamped plate problem (i.e., the first eigenvalue of the
biharmonic operator with the Dirichlet and Neumann boundary
conditions).}

\end{itemize} For this conjecture, Nadirashvili \cite{nsn} solved the
case $n=2$ while Ashbaugh and Benguria \cite{ab} solved the case
$n=3$. However, the case $n\geq4$ still remains open.

\item For a bounded domain $\Omega\subset\mathbb{R}^{n}$ with smooth boundary $\partial\Omega$,
$n\geq2$, Chasman \cite{lmc1} considered the following eigenvalue
problem of free plate
\begin{eqnarray} \label{1-2}
 \left\{
\begin{array}{lll}
\Delta^{2}u-\tau\Delta u= \Lambda u \qquad \qquad &\mathrm{in} ~ \Omega,\\
\frac{\partial^{2}u}{\partial\vec{v}^{2}}=0\qquad \qquad
&\mathrm{on} ~
\partial \Omega,\\
\tau\frac{\partial
u}{\partial\vec{v}}-\mathrm{div}_{\partial\Omega}\left({\mathrm{Proj}}_{\partial\Omega}\left[(D^{2}u)\vec{v}\right]\right)-\frac{\partial\Delta
u}{\partial\vec{v}}=0 \qquad \qquad &\mathrm{on} ~
\partial \Omega,
\end{array}
\right.
\end{eqnarray}
where $\tau\in\mathbb{R}$, $\Delta$ is the Lapacian, $\vec{v}$
denotes the outward unit normal vector of $\partial\Omega$,
$\Delta^{2}$ is the biharmonic operator in $\Omega$,
$\mathrm{div}_{\partial\Omega}$ is the surface divergence on
$\partial\Omega$, the operator ${\mathrm{Proj}}_{\partial\Omega}$
projects onto the space tangent to $\partial\Omega$, and $D^{2}u$
denotes the Hessian matrix. Physically, when $n=2$, $\Omega$ is the
shape of a homogeneous, isotropic plate, and the parameter $\tau$ is
the ratio of lateral tension to flexural rigidity of the plate.
Positive $\tau$ corresponds to a plate under tension, while negative
$\tau$ gives us a plate under compression. Chasman \cite[Section
4]{lmc1} proved that if $\tau\geq0$, the operator
$\Delta^{2}-\tau\Delta$ in the boundary value problem\footnote{ BVP
for short.} (\ref{1-2}) has a discrete spectrum and all the
eigenvalues, with finite multiplicity, in this spectrum can be
listed non-decreasingly as follows
\begin{eqnarray*}
0=\Lambda_{1}(\Omega)\leq\Lambda_{2}(\Omega)\leq\Lambda_{3}(\Omega)\leq\cdots\uparrow\infty.
\end{eqnarray*}
She also showed that among all domains with fixed volume, the lowest
nonzero
 eigenvalue $\Lambda_{2}(\Omega)$ for a free plate under
tension (i.e., $\tau>0$) is maximized by a ball (see \cite[Theorem
1]{lmc1}). Later, Chasman
 considered the following eigenvalue problem of free
plate under tension and \emph{with nonzero Poisson's ratio}
\begin{eqnarray} \label{1-3}
 \left\{
\begin{array}{lll}
\Delta^{2}u-\tau\Delta u= \Gamma u \qquad \qquad &\mathrm{in} ~ \Omega,\\
(1-\sigma)\frac{\partial^{2}u}{\partial\vec{v}^{2}}+\sigma\Delta
u=0\qquad \qquad &\mathrm{on} ~
\partial \Omega,\\
\tau\frac{\partial
u}{\partial\vec{v}}-(1-\sigma)\mathrm{div}_{\partial\Omega}\left({\mathrm{Proj}}_{\partial\Omega}\left[(D^{2}u)\vec{v}\right]\right)-\frac{\partial\Delta
u}{\partial\vec{v}}=0 \qquad \qquad &\mathrm{on} ~
\partial \Omega,
\end{array}
\right.
\end{eqnarray}
where $\sigma$ is the Poisson's ratio\footnote{ For the physical
explanation of Poisson's ratio $\sigma$ in the BVP (\ref{1-3}), see
the footnote on the $2^{\mathrm{nd}}$ page of \cite{ms}.} and other
same symbols have the same meanings as those in (\ref{1-2}).
Clearly, if $\sigma=0$ then the BVP (\ref{1-3}) degenerates into the
BVP (\ref{1-2}). Naturally, one might ask ``Whether the
isoperimetric inequality for the first nonzero eigenvalue
$\Lambda_{2}(\Omega)$ in the BVP (\ref{1-2}) can be improved to the
case of the BVP (\ref{1-3}) or not? Is this possible improvement too
direct and without any difficulty?". Chasman \cite{lmc2} gave an
answer to these questions in details. In fact, she (see
\cite[Section 4]{lmc2}) explained that if $\tau\geq0$ and
$\sigma\in(-1/(n-1),1)$, the operator $\Delta^{2}-\tau\Delta$ in the
BVP (\ref{1-3}) has a discrete spectrum and all the eigenvalues,
with finite multiplicity, in this spectrum can be listed
non-decreasingly as follows
\begin{eqnarray*}
0=\Gamma_{1}(\Omega)\leq\Gamma_{2}(\Omega)\leq\Gamma_{3}(\Omega)\leq\cdots\uparrow\infty.
\end{eqnarray*}
She also proved that  the ball with the same volume maximizes
$\Gamma_{2}(\Omega)$ if the free plate is under tension and  one of
the followings holds:

\quad (1) $n=2$ and $\sigma>-51/97$ or
$\tau\geq3(1-\sigma)/(\sigma+1)$,

\quad  (2) $n=3$,

\quad  (3) $n\geq4$ and $\sigma\leq0$ or $\tau\geq(n+2)/2$.

However, numerical and analytic evidences suggest that this fact
should hold for $\tau>0$, $\sigma\in(-1/(n-1),1)$ -- see
\cite[Section 8]{lmc2} for details. Based on this, Chasman
\cite{lmc2} conjectured:

\begin{itemize}
\item \emph{Among all domains with fixed volume, the lowest nonzero
 eigenvalue $\Gamma_{2}(\Omega)$ for a free plate under
tension, with Poisson's ratio $\sigma\in(-1/(n-1),1)$, is maximized
by a ball}.
\end{itemize}
This conjecture is open, and the best partial answers so far are due
to Chasman \cite{lmc1,lmc2}.

\item Buoso and  Provenzano \cite{bp} considered a Steklov-type
eigenvalue problem therein and showed that among all bounded
Euclidean domains of class $C^{1}$ with fixed measure, the ball
maximizes the first positive eigenvalue.

\end{itemize}
Except the above isoperimetric results, much less is known for the
biharmonic operator. The purpose of this paper is trying to get a
\emph{new} isoperimetric inequality for the biharmonic operator in a
BVP having \emph{physical background}.

Throughout this paper, let $\Omega\subset\mathbb{R}^{n}$ be a
bounded domain of class $C^{1}$, $n\geq2$. Inspired by Chasman's
work \cite{lmc1,lmc2}, we consider the following Steklov-type
eigenvalue problem
\begin{eqnarray} \label{1-4}
 \left\{
\begin{array}{lll}
\Delta^{2}u-\tau\Delta u= 0 \qquad \qquad &\mathrm{in} ~ \Omega,\\
(1-\sigma)\frac{\partial^{2}u}{\partial\vec{v}^{2}}+\sigma\Delta
u=0\qquad \qquad &\mathrm{on} ~
\partial \Omega,\\
\tau\frac{\partial
u}{\partial\vec{v}}-(1-\sigma)\mathrm{div}_{\partial\Omega}\left({\mathrm{Proj}}_{\partial\Omega}\left[(D^{2}u)\vec{v}\right]\right)-\frac{\partial\Delta
u}{\partial\vec{v}}=\lambda u \qquad \qquad &\mathrm{on} ~
\partial \Omega,
\end{array}
\right.
\end{eqnarray}
where $\sigma\in\mathbb{R}$ and other same symbols have the same
meanings as those in (\ref{1-2}). For this BVP, if $\tau>0$ and
$\sigma\in(-1/(n-1),1)$, then it only has discrete spectrum and all
the eigenvalues, with finite multiplicity, in this spectrum can be
listed non-decreasingly as follows (see Section \ref{sec2} for
details)
\begin{eqnarray*}
0=\lambda_{1}(\Omega)<\lambda_{2}(\Omega)\leq\lambda_{3}(\Omega)\leq\cdots\uparrow\infty.
\end{eqnarray*}
Moreover, we can prove:
\begin{theorem} \label{maintheorem}
For the Steklov-type problem (\ref{1-4}), if $\tau>0$ and
$\sigma\in(-1/(n-1),1)$, then
$\lambda_{2}(\Omega)\leq\lambda_{2}(\Omega^{\ast})$, where
$\Omega^{\ast}$ is the Euclidean ball having the same measure as
$\Omega$.
\end{theorem}

\begin{remark}
\rm{ (1) Chasman's experience \cite{lmc2} shows that for the free
plate problem of the biharmonic operator, there exists the essential
difference between the nonzero Poisson's ratio case and the null
case, i.e., the spectral isoperimetric result of the BVP (\ref{1-2})
cannot be simply extended to the case of BVP (\ref{1-3}). This is
exactly the motivation why we consider the boundary conditions given
in (\ref{1-4}) for the equation $\Delta^{2}u-\tau\Delta u= 0$ in
$\Omega$. \\
(2) Clearly, if the parameter $\sigma=0$, then the BVP (\ref{1-4})
becomes the one considered in \cite{bp}, and correspondingly, the
isoperimetric inequality shown in Theorem \ref{maintheorem} here
degenerates into the one obtained in \cite[Corollary 5.20]{bp}.
Besides, one can see \cite[Section 2]{bp} for the physical
background of the BVP (\ref{1-4}) in the case $\sigma=0$ and $n=2$, where it was used to describe the transverse vibrations
of a thin plate in the theory of linear elasticity. \\
 (3) As shown in Subsection \ref{subs2-3}, we would like to show
 that if additionally $\Omega$ is of
class $C^{2}$, then eigenvalues and eigenfunctions of the
Steklov-type problem
 (\ref{1-4}) can be converged by eigenvalues and eigenfunctions of
 the eigenvalue problem of the operator $\Delta^{2}-\tau\Delta$
 subject to Neumann boundary conditions. This fact gives a further
 interpretation of problem
 (\ref{1-4}) as the equation of a free
vibrating plate (under tension and with nonzero Poisson's ratio)
whose mass is concentrated at the boundary in the case of domains of
class $C^{2}$.
 \\
(4) The BVP (\ref{1-4}) has already been considered in \cite{dmwxy}
by the corresponding author, Prof. J. Mao, with his collaborators,
and a lower bound for the sums of the reciprocals of the first $n$
nonzero eigenvalues $\lambda_{i}(\Omega)$ has been obtained (see
\cite[Theorem 1.5]{dmwxy}).
 }
\end{remark}

\section{Analysis of the spectrum: characterization and asymptotic
behavior} \label{sec2}
\renewcommand{\thesection}{\arabic{section}}
\renewcommand{\theequation}{\thesection.\arabic{equation}}
\setcounter{equation}{0}

\subsection{Characterization}\label{subsec2-1}

First, we would like to show that the boundary conditions in
(\ref{1-4}) are reasonable. In order to explain clearly, we consider
a slight more general version of the problem (\ref{1-4}) as follows
\begin{eqnarray} \label{2-1-1}
 \left\{
\begin{array}{lll}
\Delta^{2}u-\tau\Delta u= 0 \qquad \qquad &\mathrm{in} ~ \Omega,\\
(1-\sigma)\frac{\partial^{2}u}{\partial\vec{v}^{2}}+\sigma\Delta
u=0\qquad \qquad &\mathrm{on} ~
\partial \Omega,\\
\tau\frac{\partial
u}{\partial\vec{v}}-(1-\sigma)\mathrm{div}_{\partial\Omega}\left({\mathrm{Proj}}_{\partial\Omega}\left[(D^{2}u)\vec{v}\right]\right)-\frac{\partial\Delta
u}{\partial\vec{v}}=\lambda\rho u \qquad \qquad &\mathrm{on} ~
\partial \Omega,
\end{array}
\right.
\end{eqnarray}
where the positive weight $\rho\in L^{\infty}(\partial\Omega)$
denotes a mass density.
 Now, consider the weak eigenvalue equation
for eigenfunction $u$ with the eigenvalue $\lambda\in \mathbb{R}$
and choose some test function $\phi\in H^{2}(\Omega)$, one has
\begin{eqnarray} \label{weq}
\int_{\Omega}\left[(1-\sigma)D^{2}u:D^{2}\phi+\sigma\Delta u \Delta
\phi+\tau \nabla u\cdot \nabla \phi
\right]dx=\lambda\int_{\partial\Omega}\rho u\phi dS,
\end{eqnarray}
where
\begin{equation*}
D^{2}u:D^{2}\phi=\sum\limits^{n}_{i,j=1}\frac{\partial^{2}u}{\partial
x_{i}\partial x_{j}}\frac{\partial^{2}\phi}{\partial x_{i}\partial
x_{j}},
\end{equation*}
 with $\{x_{i}\}_{1\leq i\leq n}$ Cartesian
coordinates of $\mathbb{R}^n$, denotes the Frobenius product,
$\nabla$ is the gradient operator, and $dx$, $dS$ are volume
densities of $\Omega$ and $\partial\Omega$ respectively. By the
divergence theorem, one has
\begin{eqnarray} \label{2}
\int_{\Omega}\nabla u\cdot \nabla\phi
dx=\int_{\partial\Omega}\phi\frac{\partial u}{\partial\vec{v}}
dS-\int_{\Omega}\phi(\Delta u)dx,
\end{eqnarray}
which implies the Hessian term becomes
\begin{eqnarray*}
\begin{split}
\int_{\Omega}(1-\sigma)D^{2}u:D^{2}\phi
dx=&(1-\sigma)\sum\limits^{n}_{i,j=1}\int_{\Omega}\frac{\partial^{2}u}{\partial
x_{i}\partial x_{j}}\frac{\partial^{2}\phi}{\partial x_{i}\partial
x_{j}}dx
=(1-\sigma)\sum\limits^{n}_{j=1}\int_{\Omega}D(u_{x_{j}})\cdot D(\phi_{x_{j}})dx\\
=&(1-\sigma)\left[\sum\limits^{n}_{j=1}\int_{\partial\Omega}\phi_{x_{j}}\frac{\partial(u_{x_{j}})}{\partial\vec{v}}dS-\sum\limits^{n}_{j=1}\int_{\Omega}\Delta(u_{x_{j}})\cdot \phi_{x_{j}}dx\right]\\
=&(1-\sigma)\Bigg{\{}\int_{\partial\Omega}\left[\frac{\partial
\phi}{\partial\vec{v}}\cdot\frac{\partial^{2}u}{\partial\vec{v}^{2}}-\phi\cdot
\mathrm{div}_{\partial\Omega}(D^{2}u\cdot
\vec{v})-\phi\frac{\partial(\Delta
u)}{\partial\vec{v}}\right]dS\\
&\qquad +\int_{\Omega}\phi\cdot \Delta^{2}udx\Bigg{\}}.
\end{split}
\end{eqnarray*}
Define the tangential divergence $\mathrm{div}_{\partial\Omega}$ of
a vector field $F$ as $\mathrm{div}_{\partial\Omega}F=\mathrm{div}
F|_{\partial\Omega}-(DF\cdot\vec{v})\cdot\vec{v}$. Applying the
divergence theorem twice, we can get
\begin{eqnarray*}
\begin{split}
\int_{\Omega}\sigma\Delta u\cdot\Delta \phi dx=&\sigma\int_{\partial\Omega}\Delta u\frac{\partial \phi}{\partial\vec{v}}dS-\sigma\int_{\Omega}D(\Delta u)\cdot D\phi dx\\
=&\sigma\int_{\partial\Omega}\left[\Delta u\frac{\partial
\phi}{\partial\vec{v}}-\phi\frac{\partial(\Delta
u)}{\partial\vec{v}}\right]dS+\sigma\int_{\Omega}\Delta^{2}u\cdot
\phi dx.
\end{split}
\end{eqnarray*}
Therefore, the weak eigenvalue equation (\ref{weq}) can be written
as
\begin{eqnarray*}
\begin{split}
&\int_{\Omega}\phi(\Delta^{2}u-\tau\Delta u)dx+\int_{\partial\Omega}\frac{\partial \phi}{\partial\vec{v}}\left[(1-\sigma)\frac{\partial^{2}u}{\partial \vec{v}^{2}}+\sigma\Delta u\right]dS+\\
&\qquad \int_{\partial\Omega}\phi\left[\tau\frac{\partial
u}{\partial\vec{v}}-\frac{\partial\Delta
u}{\partial\vec{v}}-(1-\sigma)\mathrm{div}_{\partial\Omega}(D^{2}\cdot
\vec{v})-\lambda\rho u\right]dS=0,
\end{split}
\end{eqnarray*}
which holds for all $\phi\in H^{2}(\Omega)$ and, of course, implies
the BVP (\ref{2-1-1}) directly.

\subsection{Analysis of the spectrum of the Steklov-type eigenvalue problem (\ref{2-1-1})}

Let
\begin{eqnarray}  \label{add-1-1}
\rho\in\mathcal{R}^{S}:=\left\{\rho\in
L^{\infty}(\partial\Omega)|\mathrm{essinf}_{x\in\partial\Omega}\rho(x)>0\right\}.
\end{eqnarray}
As shown in Subsection \ref{subsec2-1}, the weak formulation of BVP
(\ref{2-1-1}) is given by (\ref{weq}). Clearly, $\lambda=0$ is an
eigenvalue whose eigenfunctions are nonzero constant functions.
Hence, we need to consider the problem in the quotient space
$H^{2}(\Omega)/\mathbb{R}$. Let $\mathcal{J}^{S}_{\rho}$ be a
continuous embedding of $L^{2}(\partial\Omega)$ into
$H^{2}(\Omega)'$ defined by
 \begin{eqnarray*}
\mathcal{J}^{S}_{\rho}[u][\phi]:=\int_{\partial\Omega}\rho u\phi
dS,\qquad \forall u\in L^{2}(\partial\Omega),\phi\in H^{2}(\Omega).
\end{eqnarray*}
 Set
\begin{equation} \label{def-s1}
 H^{2,S}_{\rho}(\Omega):=\left\{ u\in H^{2}(\Omega)\Bigg{|}\int_{\partial\Omega}\rho u dS=0\right\}.
\end{equation}
In  $H^{2}(\Omega)$, consider the following bilinear form
\begin{eqnarray} \label{5}
\langle
u,\phi\rangle=\int_{\Omega}\left[(1-\sigma)D^{2}u:D^{2}\phi+\sigma\Delta
u\cdot\Delta\phi+\tau\nabla u\cdot\nabla\phi\right] dx,
\end{eqnarray}
which, by applying Poincar\'{e}-Wirtinger inequality, turns out to
be a scalar product on $H^{2,S}_{\rho}(\Omega)$. Now, we make an
agreement as follows:
\begin{itemize}

\item \emph{In the sequel, we shall treat $H^{2,S}_{\rho}(\Omega)$ as
the functional space defined by (\ref{def-s1}) and endowed with the
form (\ref{5}).}

\end{itemize}
Define a set $F(\Omega)$ as $F(\Omega):=\{G\in
H^{2}(\Omega)'|G[1]=0\}$. Now, one can define an operator
$\mathcal{P}^{S}_{\rho}:H^{2,S}_{\rho}\mapsto F(\Omega)$ given by
\begin{eqnarray} \label{6}
\mathcal{P}^{S}_{\rho}[u][v]:=\int_{\Omega}\left[(1-\sigma)D^{2}u:D^{2}v+\sigma\Delta
u\cdot\Delta v+\tau\nabla u\cdot\nabla v\right]dx,
\end{eqnarray}
where $u\in H^{2,S}_{\rho}(\Omega)$, $v\in H^{2}(\Omega)$. It is not
hard to know that $\mathcal{P}^{S}_{\rho}$ is a homeomorphism.
Define an operator $\pi^{S}_{\rho}:H^{2}(\Omega)\mapsto
H^{2,S}_{\rho}$ as follows
\begin{eqnarray} \label{7-1}
\pi^{S}_{\rho}[u]:=u-\frac{\int_{\partial\Omega}\rho
udS}{\int_{\partial\Omega}\rho dS}.
\end{eqnarray}
Consider the space $H^{2}(\Omega)/\mathbb{R}$ equipped with the
bilinear form induced by (\ref{5}), which is exactly a Hilbert
space, and then one can define a map
$\pi^{\sharp,S}_{\rho}:H^{2}(\Omega)/\mathbb{R}\mapsto
H^{2,S}_{\rho}(\Omega)$ determined by the equality
$\pi^{S}_{\rho}=\pi^{\sharp,S}_{\rho}\circ p$, with $p$ a canonical
projection of $H^{2}(\Omega)$ onto $H^{2}(\Omega)/\mathbb{R}$. It is
not hard to know that $\pi^{\sharp,S}_{\rho}$ is a homeomorphism. We
can define a differential operator $T^{S}_{\rho}$ on
$H^{2}(\Omega)/\mathbb{R}$ as follows
\begin{eqnarray} \label{8}
T^{S}_{\rho}:=(\pi^{\sharp,S}_{\rho})^{-1}\circ(\mathcal{P}^{S}_{\rho})^{-1}\circ\mathcal{J}^{S}_{\rho}\circ
\mathrm{Tr}\circ\pi^{\sharp,S}_{\rho},
\end{eqnarray}
with $\mathrm{Tr}$ the trace operator acting from $H^{2}(\Omega)$ to
$L^{2}(\partial\Omega)$. The operator $T^{S}_{\rho}$ can be shown to
be a nonnegative compact self-adjoint in $H^{2}(\Omega)/\mathbb{R}$
under the suitable assumption. However, before that, we need the
following fact:

\begin{itemize}
\item \emph{For any function $u\in H^{2}(\Omega)$, we have the sharp bound
$(\Delta u)^{2} \leq n|D^{2}u|^{2}$.}
\end{itemize}
This fact can be directly obtained by the Cauchy-Schwarz inequality.
Besides, for the operator $T^{S}_{\rho}$, it is easy to find the
following fact:

\begin{itemize}
\item The pair $(\lambda,u)$ of set $(\mathbb{R}\setminus\{0\})\times
(H^{2,S}_{\rho}(\Omega)\setminus\{0\})$ satisfies
 (\ref{weq}) if and only if $\lambda\neq0$ and the pair $(\lambda^{-1},p[u])$ of set $\mathbb{R}\times((H^{2}(\Omega)/\mathbb{R})\setminus\{0\}) $
 satisfies the equation $
   \lambda^{-1}p[u]=T^{S}_{\rho}p[u]$.
\end{itemize}
Now, we have:

\begin{theorem} \label{maintheorem-1}
Assume that $\tau>0$, $\sigma\in(-1/(n-1),1)$. The operator
$T^{S}_{\rho}$ defined by (\ref{8}) is a non-negative compact
self-adjoint operator in $H^{2}(\Omega)/\mathbb{R}$, and its
eigenvalues are the reciprocals of the positive eigenvalues of
problem (\ref{weq}). In particular, the set of eigenvalues of
problem (\ref{weq}) consists the image of a sequence contained in
$[0,\infty)$ and increasing to $+\infty$. Besides, the multiplicity
of each eigenvalue is finite.
\end{theorem}

\begin{proof}
Let $u\in H^{2}(\Omega)/\mathbb{R}$. By  the orthogonal
decomposition, we have
 \begin{eqnarray*}
u=\pi^{\sharp,S}_{\rho}[u]+(\pi^{\sharp,S}_{\rho}[u])^{\perp},
\end{eqnarray*}
where $\pi^{\sharp,S}_{\rho}[u]\in H^{2,S}_{\rho}$,
$(\pi^{\sharp,S}_{\rho}[u])^{\perp}\in (H^{2,S}_{\rho})^{\perp}$.
So, for $v\in H^{2}(\Omega)/\mathbb{R}$, one has
\begin{equation*}
\begin{split}
&\langle\pi^{\sharp,S}_{\rho}[u],v\rangle_{H^{2}(\Omega)/\mathbb{R}}\\
=&\langle\pi^{\sharp,S}_{\rho}[u],\pi^{\sharp,S}_{\rho}[v]+(\pi^{\sharp,S}_{\rho}[v])^{\perp}\rangle_{H^{2}(\Omega)/\mathbb{R}}\\
=&\langle\pi^{\sharp,S}_{\rho}[u],\pi^{\sharp,S}_{\rho}[v]\rangle_{H^{2}(\Omega)/\mathbb{R}}\\
=&\langle\pi^{\sharp,S}_{\rho}[u]+(\pi^{\sharp,S}_{\rho}[u])^{\perp},\pi^{\sharp,S}_{\rho}[v]\rangle_{H^{2}(\Omega)/\mathbb{R}}\\
=&\langle
u,\pi^{\sharp,S}_{\rho}[v]\rangle_{H^{2}(\Omega)/\mathbb{R}},
\end{split}
\end{equation*}
where, of course,
$\langle\cdot,\cdot\rangle_{H^{2}(\Omega)/\mathbb{R}}$ stands for
the bilinear form in the quotient space $H^{2}(\Omega)/\mathbb{R}$
induced by (\ref{5}). So, we can deduce that
\begin{equation*}
\begin{split}
\langle T^{S}_{\rho}u,v\rangle_{H^{2}(\Omega)/\mathbb{R}}=&\left\langle(\pi^{\sharp,S}_{\rho})^{-1}\circ(\mathcal{P}^{S}_{\rho})^{-1}\circ\mathcal{J}^{S}_{\rho}\circ\mathrm{Tr}\circ\pi^{\sharp,S}_{\rho}u,v\right\rangle_{H^{2}(\Omega)/\mathbb{R}}\\
=&\langle(\mathcal{P}^{S}_{\rho})^{-1}\circ\mathcal{J}^{S}_{\rho}\circ \mathrm{Tr}\circ\pi^{\sharp,S}_{\rho}u,(\pi^{\sharp,S}_{\rho})^{-1}v\rangle_{H^{2}(\Omega)/\mathbb{R}}\\
=&\lambda^{-1}\langle\mathrm{Tr}\circ\pi^{\sharp,S}_{\rho}u,\pi^{\sharp,S}_{\rho}v\rangle_{H^{2}(\Omega)/\mathbb{R}}\\
=&\mathcal{J}^{S}_{\rho}[\mathrm{Tr}\circ\pi^{\sharp,S}_{\rho}u][\pi^{\sharp,S}_{\rho}v]\\
=&\int_{\partial\Omega}\rho\pi^{\sharp,S}u\cdot\pi^{\sharp,S}vdS.
\end{split}
\end{equation*}
Similarly, one can obtain
 \begin{eqnarray*}
\langle
u,T^{S}_{\rho}v\rangle_{H^{2}(\Omega)/\mathbb{R}}=\int_{\partial\Omega}\rho\pi^{\sharp,S}u\cdot\pi^{\sharp,S}vdS.
 \end{eqnarray*}
So,
 \begin{eqnarray*}
\langle u,T^{S}_{\rho}v\rangle_{H^{2}(\Omega)/\mathbb{R}}=\langle
T^{S}_{\rho}u,v\rangle_{H^{2}(\Omega)/\mathbb{R}},
 \end{eqnarray*}
 which implies the self-adjointness of the operator $T^{S}_{\rho}$
 directly.

 The compactness of $T^{S}_{\rho}$ can be obtained from
 the fact that the trace operator $\mathrm{Tr}:H^{1}(\Omega)\mapsto
 L^{2}(\partial\Omega)$ is compact. Since $\sigma\in(-1/(n-1),1)$
 and $(\Delta u)^{2}\leq n|D^{2}u|^{2}$, one knows $(1-\sigma)|D^{2}u|^{2}+\sigma(\Delta
 u)^{2}>0$, which, together with $\tau>0$, implies the
 nonnegativity of the operator $\mathcal{P}^{S}_{\rho}$
 defined by (\ref{6}). Naturally, the nonnegativity of the operator $T^{S}_{\rho}$ follows
 directly.
 This completes the proof of Theorem \ref{maintheorem-1}.
\end{proof}

Applying Theorem \ref{maintheorem-1} directly, one knows that if
$\tau>0$, $\sigma\in(-1/(n-1),1)$, then the problem (\ref{weq}) only
has the discrete spectrum and all its elements (i.e., eigenvalues)
can be listed non-decreasingly as follows
\begin{eqnarray*}
0=\lambda_{1}<\lambda_{2}\leq\lambda_{3}\leq
\cdots\leq\lambda_{n}\uparrow+\infty.
\end{eqnarray*}
In fact, it is easy to check that $\lambda=0$ is an eigenvalue of
the problem (\ref{weq}) with constant functions as its
eigenfunctions. In contrast, suppose now $u$ is an eigenfunction of
the eigenvalue $\lambda=0$, and then we have
\begin{eqnarray*}
\int_{\Omega}\left[(1-\sigma)|D^{2}u|^{2}+\sigma|\Delta
u|^{2}+\tau|\nabla u|^{2}\right]dx=0,
\end{eqnarray*}
where
$|D^{2}u|^{2}:=\sum^{n}_{i,j=1}\left(\frac{\partial^{2}u}{\partial
x_{i}\partial x_{j}}\right)^{2}$. So, we have $\nabla u=0$, which
implies $u$ is constant. Moreover, by Courant's principle, it is not
hard to know that the eigenvalue $\lambda=0$ should have
multiplicity $1$.

As inspired by the free plate problem with nonzero Poisson's ratio
discussed by Chasman \cite{lmc2}, we know that the bilinear form
defined by (\ref{6}) might be coercive for $\sigma$ not in
$(-1/(n-1),1)$ if imposing some restrictions on $\tau$. For
instance, if $\tau>0$ and $\sigma=1$, then the bilinear form defined
(\ref{6}) becomes
 \begin{eqnarray*}
\int_{\Omega}\left[\Delta u\cdot\Delta v+\tau\nabla u\cdot\nabla
v\right]dx,
 \end{eqnarray*}
which is obviously coercive. Besides, in this setting, the problem
(\ref{weq}) degenerates into
\begin{eqnarray*}
\int_{\Omega}\left[\Delta u \Delta \phi+\tau \nabla u\cdot \nabla
\phi \right]dx=\lambda\int_{\partial\Omega}\rho u\phi dS,
\end{eqnarray*}
with its strengthened version
\begin{eqnarray*}
 \left\{
\begin{array}{lll}
\Delta^{2}u-\tau\Delta u= 0 \qquad \qquad &\mathrm{in} ~ \Omega,\\
\Delta u=0\qquad \qquad &\mathrm{on} ~
\partial \Omega,\\
\tau\frac{\partial u}{\partial\vec{v}}-\frac{\partial\Delta
u}{\partial\vec{v}}=\lambda\rho u \qquad \qquad &\mathrm{on} ~
\partial \Omega.
\end{array}
\right.
\end{eqnarray*}
Following almost the same argument, it is easy to know that for the
above eigenvalue problem, the operator $\Delta^{2}-\tau\Delta$ only
has discrete spectrum (with nonnegative eigenvalues inside) provided
$\tau>0$. Although, in the situation $\tau>0$ and $\sigma=1$, for
the problem (\ref{weq}), one might also have similar conclusions to
Theorem \ref{maintheorem-1}, the boundary conditions lose the
physical background, which is not the case we really want to
discuss. Besides, note that if $\tau=0$ and $\sigma=1$, then all
harmonic functions in $H^{2}$ are eigenfunctions with eigenvalue
zero to the problem (\ref{weq}), and of course, one has an
eigenvalue of infinite multiplicity. Based on these reasons, we
prefer to study the spectral properties of the problem (\ref{weq})
under the constraint that $\tau>0$, $\sigma\in(-1/(n-1),1)$.
Besides, in this constraint, by means of variational principle, the
first nonzero eigenvalue $\lambda_{2}$ of the problem (\ref{weq})
can be characterized as follows
\begin{eqnarray*}
\lambda_{2}=\min\left\{\frac{\int_{\Omega}\left[(1-\sigma)|D^{2}u|^{2}+\sigma|\Delta
u|^{2}+\tau|\nabla u|^{2}\right]dx}{\int_{\partial\Omega}\rho u^{2}
dS}\Bigg{|}0\neq u\in H^{2}(\Omega),\int_{\partial\Omega}\rho u
dS=0\right\}.
\end{eqnarray*}

\subsection{Asymptotic behavior} \label{subs2-3}

Consider the following eigenvalue problem of the biharmonic operator
with the Neumann boundary conditions
\begin{eqnarray} \label{2-1-1}
 \left\{
\begin{array}{lll}
\Delta^{2}u-\tau\Delta u=\lambda\rho u \qquad \qquad &\mathrm{in} ~ \Omega,\\
(1-\sigma)\frac{\partial^{2}u}{\partial\vec{v}^{2}}+\sigma\Delta
u=0\qquad \qquad &\mathrm{on} ~
\partial \Omega,\\
\tau\frac{\partial
u}{\partial\vec{v}}-(1-\sigma)\mathrm{div}_{\partial\Omega}\left({\mathrm{Proj}}_{\partial\Omega}\left[(D^{2}u)\vec{v}\right]\right)-\frac{\partial\Delta
u}{\partial\vec{v}}=0 \qquad \qquad &\mathrm{on} ~
\partial \Omega,
\end{array}
\right.
\end{eqnarray}
where $\rho\in\mathcal{R}^{N}:=\left\{\rho\in
L^{\infty}(\Omega)|\mathrm{essinf}_{x\in\Omega}\rho(x)>0\right\}$
 is a positive weight. This problem arises
in the study of free vibrating plate under tension and with nonzero
Poisson's ratio. One can see Section \ref{intro} for a brief
introduction of some interesting conclusions to this eigenvalue
problem.

Define
$\Omega_{\epsilon}:=\{x\in\Omega:\mathrm{dist}(x,\partial\Omega)>\epsilon\}$,
with $\mathrm{dist}(\cdot,\cdot)$ the Euclidean distance between two
geometric objects. As shown in \cite[Subsection 3.2]{bp}, one can
fix a positive number $M>0$ and choose the family of densities
$\rho_{\epsilon}$ defined by
\begin{eqnarray} \label{7}
\rho_{\epsilon}(x)=
 \left\{
\begin{array}{ll}
\epsilon, \qquad  &\mathrm{if}~ x\in  \Omega\\
\frac{M-\epsilon|\Omega_{\epsilon}|}{|\Omega\setminus\overline{\Omega}_{\epsilon}|},
\qquad  &\mathrm{if}~ x\in
\Omega\setminus\overline{\Omega}_{\epsilon}
\end{array}
\right.
\end{eqnarray}
for $\epsilon\in(0,\epsilon_{0})$ with $\epsilon_{0}>0$ sufficiently
small. Furthermore, assume that:

\begin{itemize}

\item $\Omega$ is of class $C^2$, $\epsilon_0$ can be chosen in such a way that the
map $x\mapsto x-\epsilon\vec{v}$ is a diffeomorphism between
$\partial\Omega$ and $\partial\Omega_{\epsilon}$ for all
$\epsilon\in(0,\epsilon_0)$.

\end{itemize}
By (\ref{7}), it is not hard to check that
$\int_{\Omega}\rho_{\epsilon}dx=M$ for all
$\epsilon\in(0,\epsilon_0)$. The quantity $M$ is called \emph{the
total mass of the body} (see \cite[Subsection 3.2]{bp}).

As pointed out in Section \ref{intro}, it is easy to know that if
$\tau>0$, $\sigma\in(-1/(n-1),1)$, the eigenvalue problem
(\ref{2-1-1}) only has the discrete spectrum, its nonnegative
eigenvalues (of finite multiplicity) can be listed non-decreasingly
to infinity, and moreover, all the eigenfunctions form a Hilbert
basis of $L^{2}(\Omega)$. Now, we consider the weak formulation of
problem (\ref{2-1-1}) with density $\rho_{\epsilon}$ as follows
\begin{eqnarray} \label{11}
\int_{\Omega}\left[(1-\sigma)D^{2}u:D^{2}\varphi+\sigma\Delta
u\cdot\Delta\varphi+\tau\nabla u\cdot\nabla\varphi\right]
dx=\lambda\int_{\Omega}\rho_{\epsilon} u\varphi dx,
\end{eqnarray}
for any $\varphi\in H^{2}(\Omega)$ with the unknowns $u\in
H^{2}(\Omega)$, $\lambda\in \mathbb{R}$.
Define
 \begin{eqnarray*}
\mathcal{J}^{N}_{\rho_{\epsilon}}[u][\varphi]:=\int_{\Omega}\rho_{\epsilon}u\varphi
dx,\forall u\in L^{2}(\Omega), \qquad \varphi\in H^{2}(\Omega),
 \end{eqnarray*}
 which is a
continuous embedding from $L^{2}(\Omega)$ into $H^{2}(\Omega)^{'}$.
Set
 \begin{eqnarray*}
  H^{2,N}_{\rho_{\epsilon}}(\Omega):=\left\{u\in
H^{2}(\Omega)\Big{|}\int_{\Omega}u\rho_{\varepsilon}dx=0\right\}.
 \end{eqnarray*}
The space $H^{2,N}_{\rho_{\epsilon}}(\Omega)$ can be endowed with
the form (\ref{5}) and then this form defines on
$H^{2,N}_{\rho_{\epsilon}}(\Omega)$ a scalar product, whose induced
norm is equivalent to the standard one. Define a map
$\pi^{N}_{\rho_{\epsilon}}:H^{2}(\Omega)\mapsto
H^{2,N}_{\rho_{\epsilon}}(\Omega)$ given by
 \begin{eqnarray*}
\pi^{N}_{\rho_{\epsilon}}[u]:=u-\frac{\int_{\Omega}u\rho_{\epsilon}dx}{\int_{\Omega}\rho_{\epsilon}dx}
\end{eqnarray*}
for all $u\in H^{2}(\Omega)$. Then we can define the map
$\pi^{\sharp,N}_{\rho_{\epsilon}}:H^{2}/\mathbb{R}\mapsto
H^{2,N}_{\rho_{\epsilon}}(\Omega)$ given by the equality
$\pi^{N}_{\rho_{\epsilon}}=\pi^{\sharp,N}_{\rho_{\epsilon}}\circ p$,
which is a homeomorphism. Similar to (\ref{6}), we can define a map
$\mathcal{P}^{N}_{\rho_{\epsilon}}:H^{2,N}_{\rho_{\epsilon}}(\Omega)\mapsto
F(\Omega)$ as follows
 \begin{eqnarray*}
\mathcal{P}^{N}_{\rho_{\epsilon}}[u][v]:=\int_{\Omega}\left[(1-\sigma)D^{2}u:D^{2}v+\sigma\Delta
u\cdot\Delta v+\tau\nabla u\cdot\nabla v\right]dx,
 \end{eqnarray*}
with $u\in H^{2,N}_{\rho_{\epsilon}}(\Omega)$, $v\in H^{2}(\Omega)$,
which is a linear homeomorphism of
$H^{2,N}_{\rho_{\epsilon}}(\Omega)$ onto $F(\Omega)$. Finally, we
can define an operator
$T^{N}_{\rho_{\epsilon}}:H^{2}(\Omega)/\mathbb{R}\mapsto
H^{2}(\Omega)/\mathbb{R}$ as follows
\begin{eqnarray} \label{4-5}
T^{N}_{\rho_{\epsilon}}:=(\pi^{\sharp,N}_{\rho_{\epsilon}})^{-1}\circ(\mathcal{P}^{N}_{\rho_{\epsilon}})^{-1}\circ
\mathcal{J}^{N}_{\rho_{\epsilon}}\circ
i\circ\pi^{\sharp,N}_{\rho_{\epsilon}}.
\end{eqnarray}
For the operator $T^{N}_{\rho_{\epsilon}}$, it is easy to find the
following fact:

\begin{itemize}
\item The pair $(\lambda,u)$ of set $(\mathbb{R}\setminus\{0\})\times
(H^{2,S}_{\rho}(\Omega)\setminus\{0\})$ satisfies
 (\ref{11}) if and only if $\lambda\neq0$ and the pair $(\lambda^{-1},p[u])$ of set $\mathbb{R}\times((H^{2}(\Omega)/\mathbb{R})\setminus\{0\}) $
 satisfies the equation
 $\lambda^{-1}p[u]=T^{N}_{\rho_{\epsilon}}p[u]$.
\end{itemize}
Similar to Theorem \ref{maintheorem-1}, we have:

\begin{theorem}
Let $\Omega$ be a bounded domain in $\mathbb{R}^n$ of class $C^1$
and $\epsilon\in(0,\epsilon_{0})$, and assume that
$\sigma\in(-1/(n-1),1)$, $\tau>0$. The operator
$T^{N}_{\rho_{\epsilon}}$ defined by (\ref{4-5}) is a non-negative
compact self-adjoint operator in $H^{2}(\Omega)/\mathbb{R}$, where
the eigenvalues of $T^{N}_{\rho_{\epsilon}}$ are the reciprocals of
the positive eigenvalues $\lambda_{j}(\rho_{\epsilon})$ of the
problem (\ref{2-1-1}) for all $j\in\mathbb{N}$. In particular, the
set of eigenvalues of problem (\ref{2-1-1}) consists the image of a
sequence contained in $[0,\infty)$ and increasing to $+\infty$.
Besides, the multiplicity of each eigenvalue is finite.
\end{theorem}

We have the following spectral conclusion:

\begin{theorem}
The first eigenvalue $\lambda_{1}$ of problem (\ref{2-1-1}) is equal
to zero whose eigenfunctions are the constant functions. Moreover,
$\lambda_{2}(\rho_{\epsilon})>0$.
\end{theorem}

Now, we would like to show the asymptotic property between the
problem (\ref{weq}) and the problem (\ref{2-1-1}) if $\Omega$ is of
class $C^2$. However, in order to get that, we need to make some
preparations. In fact, if $\Omega$ is of class $C^2$, then the
Tubular Neighborhood Theorem can be used to perform computations on
the strip $\Omega\setminus\overline{\Omega}_{\epsilon}$. Moreover,
following a standard argument similar to the one shown in \cite{lp},
we can get the following conclusion:

\begin{lemma}  \label{lemma-add}
Let $\Omega$ be a bounded domain in $\mathbb{R}^{n}$ of class
$C^{2}$. Let $\rho_{\epsilon}\in\mathcal{R}^{N}$ defined by
(\ref{7}). Then we have the followings:

\begin{itemize}

\item For all $\varphi\in H^{2}(\Omega)/\mathbb{R}$, $\pi^{\sharp,N}_{\rho_{\epsilon}}[\varphi]
\rightarrow\pi^{\sharp,S}_{1}[\varphi]$ in $L^{2}(\Omega)$ (also in
$H^{2}(\Omega)$) as $\epsilon\rightarrow 0$;

\item If $u_{\epsilon}\rightharpoonup u$ in
$H^{2}(\Omega)/\mathbb{R}$, then (possibly passing to a subsequence)
$\pi^{\sharp,N}_{\rho_{\epsilon}}[u_{\epsilon}]\rightarrow\pi^{\sharp,S}_{1}[u]$
in $L^{2}(\Omega)$ as $\epsilon\rightarrow 0$;

\item Assume that $u_{\epsilon}, u,w_{\epsilon}, w\in
H^{2}(\Omega)$ are functions such that $u_{\epsilon}\rightarrow
u,w_{\epsilon}\rightarrow w$ in $L^{2}(\Omega)$,
$\mathrm{Tr}[u_{\epsilon}]\rightarrow\mathrm{Tr}[u],
\mathrm{Tr}[w_{\epsilon}]\rightarrow\mathrm{Tr}[w]$ in
$L^{2}(\partial\Omega)$ as $\epsilon\rightarrow 0$. Moreover, assume
that there exists a constant $C>0$ such that $\|\nabla
u_{\epsilon}\|_{L^{2}(\Omega)}\leq C$, $\|\nabla
w_{\epsilon}\|_{L^{2}(\Omega)}\leq C$ for all $\epsilon\in
(0,\epsilon_{0})$. Then
 \begin{eqnarray*}
\int_{\Omega}\rho_{\epsilon}(u_{\epsilon}-u)w_{\epsilon}dx\rightarrow0
 \end{eqnarray*}
 and
  \begin{eqnarray*}
\int_{\Omega}\rho_{\epsilon}(w_{\epsilon}-w)udx\rightarrow 0
 \end{eqnarray*}
  as
$\epsilon\rightarrow0$.
\end{itemize}
\end{lemma}

By applying Lemma \ref{lemma-add}, we can obtain:
\begin{theorem}
Let $\Omega$ be a bounded domain in $\mathbb{R}^{n}$ of class
$C^{2}$. Let the operators $T^{S}_{\frac{M}{|\partial\Omega|}}$ and
$T^{N}_{\rho_{\epsilon}}$ from $H^{2}(\Omega)/\mathbb{R}$ to itself
be defined as in (\ref{8}) and (\ref{4-5}), respectively. Then the
sequence
$\{T^{N}_{\rho_{\epsilon}}\}_{\epsilon\in(0,\varepsilon_{0})}$
converges in norm to $T^{S}_{\frac{M}{|\partial\Omega|}}$ as
$\epsilon\rightarrow0$.
\end{theorem}

\begin{proof}
By (\ref{7-1}), one can easily get
$$\pi^{\sharp,S}_{c}=\pi^{\sharp,S}_{1}$$
 for all $c\in\mathbb{R}$ with $c\neq 0$. We need to show that the
family of compact operators
$\{T^{N}_{\rho_{\epsilon}}\}_{\epsilon\in(0,\epsilon_{0})}$
converges compactly to the compact operator
$T^{S}_{\frac{M}{|\partial\Omega|}}$, that is to say,
\begin{eqnarray} \label{2-13}
\lim\limits_{\epsilon\rightarrow
0}\left\|\left(T^{N}_{\rho_{\epsilon}}-T^{S}_{\frac{M}{|\partial\Omega|}}\right)^{2}
\right\|_{L(H^{2}(\Omega)/\mathbb{R},H^{2}(\Omega)/\mathbb{R})}=0.
\end{eqnarray}
The operators
$\{T^{N}_{\rho_{\epsilon}}\}_{\epsilon\in(0,\epsilon_{0})}$ and
$T^{S}_{\frac{M}{|\partial\Omega|}}$ are self-adjoint, so we only
need to prove that
$\{T^{N}_{\rho_{\epsilon}}\}_{\epsilon\in(0,\varepsilon_{0})}$
converges to $T^{S}_{\frac{M}{|\partial\Omega|}}$ in norm. By
(\ref{4-5}), we know that the operator $T^{N}_{\rho_{\epsilon}}$
compactly converges to $T^{S}_{\frac{M}{|\partial\Omega|}}$ if the
following requirements are satisfied:

\begin{itemize}

\item \underline{(R1)} If $\|u_{\epsilon}\|_{H^{2}(\Omega)/\mathbb{R}}\leq C$ for all $\epsilon\in(0,\epsilon_{0})$,
then the family
$\{T^{N}_{\rho_{\epsilon}}u_{\epsilon}\}_{\epsilon\in(0,\epsilon_{0})}$
is relatively compact in $H^{2}(\Omega)/\mathbb{R}$;

\item \underline{(R2)} if $u_{\epsilon}\rightarrow u$ in $H^{2}(\Omega)/\mathbb{R}$,
then $T^{N}_{\rho_{\epsilon}}u_{\epsilon}\rightarrow
T^{S}_{\frac{M}{|\partial\Omega|}}u$ in $H^{2}(\Omega)/\mathbb{R}$.

\end{itemize}
We show \underline{(R1)} first.  For a fixed $u\in
H^{2}(\Omega)/\mathbb{R}$, by Lemma \ref{lemma-add}, we have
\begin{equation*}
\begin{split}
\lim\limits_{\epsilon\rightarrow
0}\int_{\Omega}\rho_{\epsilon}\pi^{\sharp,N}_{\rho_{\epsilon}}[u]dx=&\lim\limits_{\epsilon\rightarrow
0}
\int_{\Omega}\rho_{\epsilon}(\pi^{\sharp,N}_{\rho_{\epsilon}}[u]-\pi^{\sharp,S}_{1}[u])dx\\
&\quad+\left(\lim\limits_{\epsilon\rightarrow 0}\int_{\Omega}\rho_{\epsilon}\pi^{\sharp,S}_{1}[u]dx-\frac{M}{|\partial\Omega|}\int_{\partial\Omega}\pi^{\sharp,S}_{1}[u]dS\right)\\
&\quad+\frac{M}{|\partial\Omega|}\int_{\partial\Omega}\pi^{\sharp,S}_{1}[u]dS\\
=&\frac{M}{|\partial\Omega|}\int_{\partial\Omega}\pi^{\sharp,S}_{1}[u]dS.
\end{split}
\end{equation*}
Moreover, the equality
$(\pi^{\sharp,N}_{\rho_{\epsilon}})^{-1}\circ(\mathcal{P}^{N}_{\rho_{\epsilon}})^{-1}=
(\pi^{\sharp,S}_{1})^{-1}\circ(\mathcal{P}^{S}_{1})^{-1}$ holds,
which implies that $T^{N}_{\rho_{\epsilon}}u$ is bounded for each
$u\in H^{2}(\Omega)/\mathbb{R}$. By Banach-Steinhaus Theorem, there
exists a non-negative constant $C'$ such that
$\|T^{N}_{\rho_{\epsilon}}u\|_{L(H^{2}(\Omega)/\mathbb{R},H^{2}(\Omega)/\mathbb{R})}\leq
C'$ for all $\epsilon\in(0,\epsilon_{0})$. Moreover, since
$\|u_{\epsilon}\|_{H^{2}(\Omega)/\mathbb{R}}\leq C$ for
$\epsilon\in(0,\epsilon_{0})$,
 possibly passing to a subsequence, we have $u_{\epsilon}\rightharpoonup u$ in $H^{2}(\Omega)/\mathbb{R}$ for some
$u\in H^{2}(\Omega)/\mathbb{R}$. Hence, possibly passing to a
subsequence, $T^{N}_{\rho_{\epsilon}}u_{\epsilon}\rightharpoonup w$
in $H^{2}(\Omega)/\mathbb{R}$ as $\epsilon\rightarrow 0$. We can
show $w=T^{S}_{\frac{M}{|\partial\Omega|}}u$. Set
$w_{\epsilon}:=T^{N}_{\rho_{\epsilon}}u_{\epsilon}$. By Lemma
\ref{lemma-add}, one has
\begin{equation*}
\begin{split}
\lim\limits_{\epsilon\rightarrow
0}&\int_{\Omega}\Big{[}(1-\sigma)D^{2}(\pi^{\sharp,N}_{\rho_{\epsilon}}
[w_{\epsilon}]):D^{2}(\pi^{\sharp,N}_{\rho_{\epsilon}}[\varphi])
+\sigma\Delta(\pi^{\sharp,N}_{\rho_{\epsilon}}[w_{\epsilon}])\cdot\Delta(\pi^{\sharp,N}_{\rho_{\epsilon}}[\varphi])\\
&+\tau\nabla(\pi^{\sharp,N}_{\rho_{\epsilon}}[w_{\epsilon}])\cdot\nabla(\pi^{\sharp,N}_{\rho_{\epsilon}}[\varphi])\Big{]}dx\\
=&\int_{\Omega}\Big{[}(1-\sigma)D^{2}(\pi^{\sharp,S}_{1}[w]):D^{2}(\pi^{\sharp,S}_{1}[\varphi])+
\sigma\Delta(\pi^{\sharp,S}_{1}[w])\cdot\Delta(\pi^{\sharp,S}_{1}[\varphi])\\
&+\tau\nabla(\pi^{\sharp,S}_{1}[w])\cdot\nabla(\pi^{\sharp,S}_{1}[\varphi])\Big{]}dx
\end{split}
\end{equation*}
for all $\varphi\in H^{2}(\Omega)/\mathbb{R}$. On the other hand,
from the equality $(\mathcal{P}^{N}_{\rho_{\epsilon}}\circ
\pi^{\sharp,N}_{\rho_{\epsilon}})w_{\epsilon}=(\mathcal{J}^{N}_{\rho_{\epsilon}}\circ
i\circ\pi^{\sharp,N}_{\rho_{\epsilon}})u_{\epsilon}$, it follows
that
\begin{equation*}
\begin{split}
\int_{\Omega}&\Big{[}(1-\sigma)D^{2}(\pi^{\sharp,N}_{\rho_{\epsilon}}
[w_{\epsilon}]):D^{2}(\pi^{\sharp,N}_{\rho_{\epsilon}}[\varphi])
+\sigma\Delta(\pi^{\sharp,N}_{\rho_{\epsilon}}[w_{\epsilon}])\cdot\Delta(\pi^{\sharp,N}_{\rho_{\epsilon}}[\varphi])\\
&+\tau\nabla(\pi^{\sharp,N}_{\rho_{\epsilon}}[w_{\epsilon}])\cdot\nabla(\pi^{\sharp,N}_{\rho_{\epsilon}}[\varphi])\Big{]}dx\\
=&\int_{\Omega}\rho_{\epsilon}\pi^{\sharp,N}_{\rho_{\epsilon}}[u_{\epsilon}]\pi^{\sharp,N}_{\rho_{\epsilon}}[\varphi]dx.
\end{split}
\end{equation*}
Then, by the third claim of Lemma \ref{lemma-add}, we have
\begin{equation*}
\begin{split}
\langle
w,\varphi\rangle_{H^{2}(\Omega)/\mathbb{R}}=&\lim\limits_{\epsilon\rightarrow
0}\langle
 w_{\epsilon},\varphi\rangle_{H^{2}(\Omega)/\mathbb{R}}
=\lim\limits_{\epsilon\rightarrow
0}\int_{\Omega}\rho_{\epsilon}\pi^{\sharp,N}
_{\rho_{\epsilon}}[u_{\epsilon}]\pi^{\sharp,N}_{\rho_{\epsilon}}[\varphi]dx \\
=&\lim\limits_{\epsilon\rightarrow
0}\int_{\Omega}\rho_{\epsilon}(\pi^{\sharp,N}
_{\rho_{\epsilon}}[u_{\epsilon}]-\pi^{\sharp,S}
_{1}[u])\pi^{\sharp,N}_{\rho_{\epsilon}}[\varphi]dx \\
&+\lim\limits_{\epsilon\rightarrow
0}\int_{\Omega}\rho_{\epsilon}\pi^{\sharp,S}
_{1}[u](\pi^{\sharp,N}_{\rho_{\epsilon}}[\varphi]-\pi^{\sharp,S}
_{1}[\varphi])dx \\
&+\lim\limits_{\epsilon\rightarrow
0}\int_{\Omega}\rho_{\epsilon}\pi^{\sharp,S}
_{1}[u]\pi^{\sharp,S}_{1}[\varphi]dx \\
=&\frac{M}{|\partial\Omega|}\int_{\partial\Omega}\pi^{\sharp,S}
_{1}[u]\pi^{\sharp,S}_{1}[\varphi]dS \\
=&\left\langle
T^{S}_{\frac{M}{|\partial\Omega|}}u,\varphi\right\rangle_{H^{2}(\Omega)/\mathbb{R}},
\end{split}
\end{equation*}
and therefore  $w=T^{S}_{\frac{M}{|\partial\Omega|}}u$. Similarly,
one has
$\|w_{\epsilon}\|_{H^{2}(\Omega)/\mathbb{R}}\rightarrow\|w\|_{H^{2}(\Omega)/\mathbb{R}}$.
Then
\begin{equation*}
\begin{split}
\lim\limits_{\epsilon\rightarrow
0}\|w_{\epsilon}\|^{2}_{H^{2}(\Omega)/\mathbb{R}}
=&\lim\limits_{\epsilon\rightarrow
0}\int_{\Omega}\rho_{\epsilon}(\pi^{\sharp,N}
_{\rho_{\epsilon}}[u_{\epsilon}]-\pi^{\sharp,S}
_{1}[u])\pi^{\sharp,N}_{\rho_{\epsilon}}[w_{\epsilon}]dx\\
&+\lim\limits_{\epsilon\rightarrow
0}\int_{\Omega}\rho_{\epsilon}\pi^{\sharp,S}
_{1}[u](\pi^{\sharp,N}_{\rho_{\epsilon}}[w_{\epsilon}]-\pi^{\sharp,S}
_{1}[w_{\epsilon}])dx\\
&+\lim\limits_{\epsilon\rightarrow
0}\int_{\Omega}\rho_{\epsilon}\pi^{\sharp,S}
_{1}[u](\pi^{\sharp,S}_{1}[w_{\epsilon}]-\pi^{\sharp,S}_{1}[w])dx\\
&+\lim\limits_{\epsilon\rightarrow
0}\int_{\Omega}\rho_{\epsilon}\pi^{\sharp,S}
_{1}[u]\pi^{\sharp,S}_{1}[w]dx\\
=&\frac{M}{|\partial\Omega|}\int_{\Omega}\pi^{\sharp,S}
_{1}[u]\pi^{\sharp,S}_{1}[w]dS\\
=&\|w\|^{2}_{H^{2}(\Omega)/\mathbb{R}},
\end{split}
\end{equation*}
which finishes the proof of \underline{(R1)}.

Let $u_{\varepsilon}\rightarrow u$ in ${H^{2}(\Omega)/\mathbb{R}}$.
Then there exists some non-negative constant $C''$ such that
 \begin{eqnarray*}
\|u_{\epsilon}\|^{2}_{H^{2}(\Omega)/\mathbb{R}}\leq C{''}
 \end{eqnarray*}
 for all $\epsilon$.
Then, using a similar argument to that of the claim
\underline{(R1)}, for each sequence $\epsilon_{j}\rightarrow 0$,
possibly passing to a subsequence, we have
$T^{N}_{\rho_{\epsilon_{j}}}u_{\epsilon_{j}}\rightarrow
T^{S}_{\frac{M}{|\partial\Omega|}}u$. Since this is true for each
$\{\epsilon_{j}\}_{ j \in\mathbb{N}}$, we have the convergence for
the whole family,
i.e.,$T^{N}_{\rho_{\epsilon}}u_{\epsilon}\rightarrow
T^{S}_{\frac{M}{|\partial\Omega|}}u$, which completes the proof of
\underline{(R2)}.
\end{proof}

In the sequel, we will explore the maximum value of the fundamental
tone of problem (\ref{1-4}) and prove that when $\Omega^{\ast}$ is a
ball such that $|\Omega|=|\Omega^{\ast}|$, this maximum value can be
taken by the corresponding eigenvalue of $\Omega^{\ast}$.

\section{Eigenvalues and eigenfunctions
on the ball}
\renewcommand{\thesection}{\arabic{section}}
\renewcommand{\theequation}{\thesection.\arabic{equation}}
\setcounter{equation}{0}

When $\mathbf{B}$ is a unit ball in $\mathbb{R}^{n}$ centered at the
origin, we would like to characterize  the eigenvalues and the
corresponding eigenfunctions of (\ref{1-4}). We will use spherical
coordinates $(r,\theta)$ to do calculations, with
$\theta=(\theta_{1},\cdot\cdot\cdot,\theta_{n-1})\in\mathbb{S}^{n-1}$
and $\mathbb{S}^{n-1}$ the $(n-1)$-dimensional unit Euclidean
sphere. In fact, the corresponding coordinate transformation should
be
\begin{flalign*}
&x_{1}=r\cos(\theta_{1}),\\
&x_{2}=r\sin(\theta_{1})\cos(\theta_{2}),\\
&.\\
&.\\
&.\\
&x_{n-1}=r\sin(\theta_{1})\sin(\theta_{2})\cdot\cdot\cdot\sin(\theta_{n-2})\cos(\theta_{n-1}),\\
&x_{n}=r\sin(\theta_{1})\sin(\theta_{2})\cdot\cdot\cdot\sin(\theta_{n-2})\sin(\theta_{n-1}).
\end{flalign*}
Here,if $n>2$, then
$\theta_{1},\cdot\cdot\cdot\theta_{n-2}\in[0,\pi],\theta_{n-1}\in[0,
2\pi)$, while if $n=2$, then $\theta_{1}\in[0,2\pi)$. In spherical
coordinates, the boundary condition of problem (\ref{1-4}) can be
written as
\begin{eqnarray} \label{1-1}
 \left\{
\begin{array}{ll}
(1-\sigma)\frac{\partial^{2}u}{\partial
r^{2}}\Big{|}_{r=1}+\sigma\Delta u=0
,\\[2mm]
\tau\frac{\partial u}{\partial
r}-\frac{1-\sigma}{r^{2}}\Delta_{S}(\frac{\partial u}{\partial
r}-\frac{u}{r})-\frac{\partial\Delta u}{\partial
r}\Big{|}_{r=1}=\lambda u |_{r=1},
\end{array}
\right.
\end{eqnarray}
 where $\Delta_{S}$ is the angular part of the Laplacian. Using a
 similar argument to \cite{lmc1,lmc2}, we know that for any
 nonnegative eigenvalue $\lambda$, its eigenfunction can be written as a
product of a radial part and an angular part. Besides, the radial
part of the fundamental tone of the unit ball should be a linear
combination of Bessel functions. As we know, the ultraspherical
Bessel functions $j_{l}(z)$ of the first kind are defined as
\begin{eqnarray*}
j_{l}(z):=z^{1-\frac{n}{2}}J_{\frac{n}{2}-1+l}(z)
\end{eqnarray*}
with the modified Bessel functions $J_{v}$ solving the following
Bessel equation
\begin{eqnarray*}
z^{2}y{''}(z)+zy{'}(z)+(z^{2}-v^{2})y(z)=0.
\end{eqnarray*}
The first and the second kinds of hypersphere modified Bessel
functions $i_{l}(z)$, $k_{l}(z)$ are defined separately as follows
\begin{eqnarray*}
i_{l}(z):=z^{1-\frac{n}{2}}I_{\frac{n}{2}-1+l}(z)
\end{eqnarray*}
and
\begin{eqnarray*}
k_{l}(z):=z^{1-\frac{n}{2}}K_{\frac{n}{2}-1+l}(z),
\end{eqnarray*}
where the modified Bessel functions $I_{v}$ and $K_{v}$ are
solutions to the following modified ultraspherical Bessel equation
\begin{eqnarray*}
z^{2}y{''}(z)+zy{'}(z)+(z^{2}+v^{2})y(z)=0.
\end{eqnarray*}
By \cite[\S 9.6]{as}, one knows that $i_{l}(z)$ and all its
derivatives are positive on $(0,+\infty)$. Based on these facts, we
can prove:

\begin{theorem} \label{theorem3-1}
Let $\Omega$  be the unit ball, centered at the origin, in
$\mathbb{R}^{n}$. Any eigenfunction $u_{l}$ of the problem
(\ref{1-4}) is of the form $u_{l}(r,\theta)=R_{l}(r)Y_{l}(\theta)$,
where $Y_{l}(\theta)$ is a spherical harmonic function of some order
$l\in \mathbb{N}$ and
 \begin{eqnarray*}
R_{l}(r)=A_{l}r^{l}+B_{l}i_{l}(\sqrt{\tau}r),
 \end{eqnarray*}
 where $A_{l}$ and
$B_{l}$ are suitable constants such that
$$B_{l}=\frac{(1-\sigma)l(1-l)A_{l}}{\tau i{''}_{l}(\sqrt{\tau})+\sqrt{\tau}\sigma(n-1)i^{\prime}_{l}(\sqrt{\tau})-\sigma l(l+n-2)i_{l}(\sqrt{\tau})}$$
Moreover, the eigenvalue $\lambda_{(l)}$ associated with the
eigenfunction $u_{l}$ is delivered by formula
\begin{eqnarray}  \label{theorem3-1-1}
\begin{split}
\lambda_{(l)}&=\left(\tau i{''}_{l}(\sqrt{\tau})+\sqrt{\tau}\sigma(n-1)i{'}_{l}(\sqrt{\tau})-\sigma l(l+n-2)i_{l}(\sqrt{\tau})+(1-\sigma)l(1-l)i_{l}(\sqrt{\tau})\right)^{-1}\\
&\quad\Big{\{}-l^{2}(l+n-2)(\sigma\tau+(1-\sigma)(l-1)(\sigma l+\sigma n-3-\sigma))i_{l}(\sqrt{\tau})+\Big{[}\tau\sqrt{\tau}l\cdot\\
&\quad(\sigma n+l\sigma-2\sigma-l+1)+\sqrt{\tau}(1-\sigma)l(l-1)((l+n-2)(\sigma n-\sigma-2l+\sigma l)\\
&\quad -n+1)\Big{]}i{'}_{l}(\sqrt{\tau})+\tau l(\tau+(1-\sigma)(l+2n-3)(l-1))i{''}_{l}(\sqrt{\tau})\\
&\quad
+\tau\sqrt{\tau}(1-\sigma)l(l-1)i{'''}_{l}(\sqrt{\tau})\Big{\}}
\end{split}
\end{eqnarray}
for any $l\in \mathbb{N}$.
\end{theorem}

\begin{proof}
It is easy to know that the solution $u$ of the Steklov-type
eigenvalue problem (\ref{1-4}) in the unit ball is smooth (see,
e.g., \cite{ggs}). We divide into the argument into two cases:

\textbf{Case 1}. Assume that $\Delta u=0$.  The Laplace operator in
spherical coordinates can be written as
 \begin{eqnarray*}
\Delta=\partial_{rr}+\frac{n-1}{r}\partial_{r}+\frac{1}{r^{2}}\Delta_{S}.
\end{eqnarray*}
Separating variables so that $u=R(r)Y(\theta)$, we can obtain
\begin{eqnarray} \label{ODE-1}
\begin{split}
R{''}+\frac{n-1}{r}R{'}-\frac{l(l+n-2)}{r^{2}}R=0
\end{split}
\end{eqnarray}
with
\begin{eqnarray} \label{ODE-2}
\begin{split}
\Delta_{S}Y=-l(l+n-2)Y.
\end{split}
\end{eqnarray}
Solving the ODE (\ref{ODE-1}) yields $R(r)=ar^{l}+br^{2-n-l}$, where
$l>0$, $n\geq 2$. Besides, if $l=0$, $n=2$, $R(r)=a+b\log r$. When
$r=0$, if $b\neq0$, then $u$ blows up at $r=0$, and so we have to
impose $b=0$. Besides, the solutions of (\ref{ODE-2}) are the
spherical harmonic functions of order $l$. So, we have
 \begin{eqnarray*}
u(r,\theta)=a_{l}r^{l}Y_{l}(\theta).
 \end{eqnarray*}
for some $l\in\mathbb{N}$.

\textbf{Case 2}. Assume that $\Delta u\neq0$. Let $v=\Delta u$, and
then the Equation $\Delta^{2}u-\tau\Delta u=0$ can be written as
 \begin{eqnarray*}
\Delta v=\tau v.
 \end{eqnarray*}
Similarly, separating variables so that $v=R(r)Y(\theta)$, we have
\begin{eqnarray} \label{ODE-3}
\begin{split}
R{''}+\frac{n-1}{r}R{'}-\frac{l(l+n-2)}{r^{2}}R=\tau R,
\end{split}
\end{eqnarray}
with $Y$ satisfying (\ref{ODE-2}). For the ODE (\ref{ODE-3}), using
a similar argument to that of ODEs (\ref{ODE-1}) and (\ref{ODE-2})
yields that $v$ should be
 \begin{eqnarray*}
v(r,\theta)=b_{l_{1}}i_{l_{1}}(\sqrt{\tau}r)Y_{l_{1}}(\theta),
 \end{eqnarray*}
  for some
$l_{1}\in\mathbb{N}$. Since $v=\frac{\Delta v}{\tau}=\Delta u$, we
have
\begin{eqnarray} \label{add-xx}
\begin{split}
u(r,\theta)=\frac{b_{l_{1}}}{\tau}i_{l_{1}}(\sqrt{\tau}r)Y_{l_{1}}(\theta)-c_{l_{2}}r^{l_{2}}Y_{l_{2}}(\theta),
\end{split}
\end{eqnarray}
where $l_{2}\in\mathbb{N}$. Rewriting the boundary condition
$(1-\sigma)\frac{\partial^{2}u}{\partial
r^{2}}\Big{|}_{r=1}+\sigma\Delta u=0$ as follows
\begin{eqnarray} \label{add-x3}
0&=&b_{l_{1}}i{''}_{l_{1}}(\sqrt{\tau})Y_{l_{1}}(\theta)+\sigma(n-1)\frac{b_{l_{1}}}{\sqrt{\tau}}i{'}_{l_{1}}(\sqrt{\tau})
Y_{l_{1}}(\theta)-\sigma\frac{b_{l_{1}}}{\tau}i_{l_{1}}(\sqrt{\tau})l_{1}(l_{1}+n-2)Y_{l_{1}}(\theta)\nonumber\\
&&-l_{2}(l_{2}-1)c_{l_{2}}Y_{l_{2}}(\theta)-\sigma(n-1)l_{2}c_{l_{2}}Y_{l_{2}}(\theta)+\sigma c_{l_{2}}l_{2}(l_{2}+n-2)Y_{l_{2}}(\theta)\nonumber\\
&=&\left[\frac{l_{1}(l_{1}-1)(1-\sigma)}{\tau}i_{l_{1}}(\sqrt{\tau})+\frac{2l_{1}+\sigma(n-1)+1}{\sqrt{\tau}}i_{l_{1}+1}(\sqrt{\tau})+i_{l_{1}+2}(\sqrt{\tau})\right]
b_{l_{1}}Y_{l_{1}}(\theta)\nonumber\\
&&-c_{l_{2}}l_{2}(l_{2}-1)(1-\sigma)Y_{l_{2}}(\theta).
\end{eqnarray}
We recall that $i_{l}(z)$ and all its derivatives are positive on
$(0,+\infty)$. Combining the fact that $Y_{l_{1}}$ and $Y_{l_{2}}$
are linearly independent, we can easily obtain that $b_{l_{1}}=0$
and $l_{2}=0$ or $l_{2}=1$. Then, together with (\ref{add-xx}), it
follows that
\begin{eqnarray*}
\begin{split}
u_{l}(r,\theta)=[-c_{l}r^{l}+\frac{b_{l}}{\tau}i_{l}(\sqrt{\tau}r)]Y_{l}(\theta).
\end{split}
\end{eqnarray*}
Let $A_{l}=-c_{l}$, $B_{l}=\frac{b_{l}}{\tau}$, and then
$u_{l}(r,\theta)$ can be rewritten as
\begin{eqnarray} \label{add-x4}
\begin{split}
u_{l}(r,\theta)=[A_{l}r^{l}+B_{l}i_{l}(\sqrt{\tau}r)]Y_{l}(\theta).
\end{split}
\end{eqnarray}
By (\ref{add-x3}) and a direct calculation, one has
\begin{eqnarray}  \label{add-x5}
\begin{split}
B_{l}=-\frac{(1-\sigma)l(l-1)}{\tau i{''}_{l}(\sqrt{\tau})
+\sqrt{\tau}\sigma(n-1)i{'}_{l}(\sqrt{\tau})-\sigma
l(l+n-2)i_{l}(\sqrt{\tau})}A_{l}.
\end{split}
\end{eqnarray}
We know that $u_{l}$ given by (\ref{add-x4}) is an eigenfunction of
the eigenvalue problem (\ref{1-4}) on the unit ball. Combining the
boundary condition
 \begin{eqnarray*}
\tau\frac{\partial u}{\partial
r}-\frac{1-\sigma}{r^{2}}\Delta_{S}\left(\frac{\partial u}{\partial
r}-\frac{u}{r}\right)-\frac{\partial\Delta u}{\partial
r}\Big{|}_{r=1}=\lambda u |_{r=1}
 \end{eqnarray*}
 (\ref{add-x4}) yields
\begin{eqnarray*}
\begin{split}
&A_{l}l[\tau+(1-\sigma)(l+n-2)(l-1)]+B_{l}\{-l(l+n-2)(3-\sigma)i_{l}(\sqrt{\tau})+\sqrt{\tau}[\tau+\\
&(2-\sigma)l(l+n-2)+n-1]i^{\prime}_{l}(\sqrt{\tau})-\tau(n-1)i{''}_{l}(\sqrt{\tau})-
\tau\sqrt{\tau}i{'''}(\sqrt{\tau})\}\\
=&\lambda_{(l)}[A_{l}+B_{l}i_{l}(\sqrt{\tau})]
\end{split}
\end{eqnarray*}
The conclusion (\ref{theorem3-1-1}) follows directly by substituting
(\ref{add-x5}) into the above equality.
\end{proof}

Now, we can give the characterization to the first nonzero
eigenvalue and also its eigenfunctions.

\begin{theorem} \label{theorem3-2}
Let $\Omega=\mathbf{B}_{1}$ be the unit ball in $\mathbb{R}^{n}$
centered at the origin. The first positive eigenvalue of the
eigenvalue problem (\ref{1-4}) is $\lambda_{2}=\lambda_{(1)}=\tau$.
The corresponding eigenspace is generated by
$\{x_1,x_2,...,x_{n}\}$.
\end{theorem}

\begin{proof}
By Theorem \ref{theorem3-1}, we can get
$0=\lambda_{(0)}<\tau=\lambda_{(1)}$.  Recall the recursive relation
of some known hypersphere Bessel functions (see, e.g., \cite[p.
376]{as})
\begin{flalign*}
&i_l(z)=\frac{2+2l}{z}i_{l+1}(z)+i_{l+2}(z)\\
&i_{l}{'}(z)=\frac{l}{z}i_{l}(z)+i_{l+1}(z),\\
&i_{l}{''}(z)=\frac{l(l-1)}{z^{2}}i_{l}(z)+\frac{2l+1}{z}i_{l+1}(z)+i_{l+2}(z),\\
&i_{l}{'''}(z)=\frac{l(l-1)(l-2)}{z^{3}}i_{l}(z)+\frac{3l^{2}}{z^{2}}i_{l+1}(z)+\frac{3l+3}{z}i_{l+2}(z)+i_{l+3}(z).
\end{flalign*}
Let $D_{(l)}$ be the denominator of the RHS of (\ref{theorem3-1-1})
and $N_{(l)}$ be the corresponding numerator, i.e.,
$\lambda_{(l)}=\frac{N_{(l)}}{D_{(l)}}$. By the recursive formula,
we can calculate the denominator $D_{(l)}$ as follows
\begin{eqnarray*}
\begin{split}
D_{(l)}=&\tau i{''}_{l}(\sqrt{\tau})+\sqrt{\tau}\sigma(n-1)i{'}_l(\sqrt{\tau})-\sigma l(l+n-2)i_l(\sqrt{\tau})-(1-\sigma)l(l-1)i_l(\sqrt{\tau})\\
=&\tau\left(\frac{(l-1)l}{\tau}i_{l}(\sqrt{\tau})+\frac{2l+1}{\sqrt{\tau}}i_{l+1}(\sqrt{\tau})+i_{l+2}(\sqrt{\tau})\right)+\sqrt{\tau}\sigma(n-1)\left(\frac{l}{\sqrt{\tau}}
i_l(\sqrt{\tau})+i_{l+1}(\sqrt{\tau})\right)\\
&-\sigma l(l+n-2)i_l(\sqrt{\tau})-(1-\sigma)l(l-1)i_l(\sqrt{\tau})\\
=&\sqrt{\tau}(2l+\sigma n+1-\sigma)i_{l+1}(\sqrt{\tau})+\tau
i_{l+2}(\sqrt{\tau}).
\end{split}
\end{eqnarray*}
Moreover, the numerator $N_{(l)}$ can be computed as follows
\begin{eqnarray*}
\begin{split}
N_{(l)}=&-l^{2}(l+n-2)(\sigma\tau+(1-\sigma)(l-1)(\sigma l+\sigma n-3-\sigma))i_{l}(\sqrt{\tau})+\Big{[}\tau\sqrt{\tau}l\cdot\\
&(\sigma n+l\sigma-2\sigma-l+1)+\sqrt{\tau}(1-\sigma)l(l-1)((l+n-2)(\sigma n-\sigma-2l+\sigma l)-n+1)\Big{]}\cdot \\
&i{'}_{l}(\sqrt{\tau})+\tau l(\tau+(1-\sigma)(l+2n-3)(l-1))i{''}_{l}(\sqrt{\tau})+\tau\sqrt{\tau}(1-\sigma)l(l-1)i{'''}_{l}(\sqrt{\tau})\\
=&-l^{2}(l+n-2)(\sigma\tau+(1-\sigma)(l-1)(\sigma l+\sigma n-3-\sigma))i_{l}(\sqrt{\tau})+\\
&\qquad \Big{[}\tau\sqrt{\tau}l(\sigma n+l\sigma-2\sigma-l+1)\\
&+\sqrt{\tau}(1-\sigma)l(l-1)((l+n-2)(\sigma n-\sigma-2l+\sigma l)-n+1)\Big{]}\left(\frac{l}{\sqrt{\tau}}i_{l}(\sqrt{\tau})+i_{l+1}(\sqrt{\tau})\right)\\
&+\tau l(\tau+(1-\sigma)(l+2n-3)(l-1))\left(\frac{l(l-1)}{\tau}i_{l}(\sqrt{\tau})+\frac{2l+1}{\sqrt{\tau}}i_{l+1}(\sqrt{\tau})+i_{l+2}(\sqrt{\tau})\right)\\
&+\tau\sqrt{\tau}(1-\sigma)l(l-1)\Bigg{(}\frac{l(l-1)(l-2)}{\tau\sqrt{\tau}}i_{l}(\sqrt{\tau})+\frac{3l^{2}}{\tau}i_{l+1}
(\sqrt{\tau})+\\
&\qquad
 \frac{3l+3}{\sqrt{\tau}}i_{l+2}(\sqrt{\tau})+i_{l+3}(\sqrt{\tau})\Bigg{)}\\
=&[\tau\sqrt{\tau} l(\sigma n+\sigma l-2\sigma+l+2)+\sqrt{\tau}(1-\sigma)l(l-1)((l+n-2)(\sigma n-\sigma+\sigma l+1)+\\
&3l^2+2nl-2l)]i_{l+1}(\sqrt{\tau})+[\tau
l(\tau+(1-\sigma)(l-1)(4l+2n))]i_{l+2}(\sqrt{\tau})+\\
&\qquad \tau\sqrt{\tau}(1-\sigma)l(l-1)i_{l+3}(\sqrt{\tau}).
\end{split}
\end{eqnarray*}
Then applying the recursion formula again and the conclusion
(\ref{theorem3-1-1}), we have
\begin{eqnarray*}
\begin{split}
\lambda_{(l)}=&\frac{N_{(l)}}{D_{(l)}}
=\left(\sqrt{\tau}(2l+\sigma n+1-\sigma)i_{l+1}(\sqrt{\tau})+\tau i_{l+2}(\sqrt{\tau})\right)^{-1}\Big{\{}[\tau\sqrt{\tau} l(\sigma n+\sigma l-2\sigma+l+2)\\
&+\sqrt{\tau}(1-\sigma)l(l-1)((l+n-2)(\sigma n-\sigma+\sigma l+1)+3l^2+2nl-2l)]i_{l+1}(\sqrt{\tau})+[\tau l(\tau\\
&+(1-\sigma)(l-1)(4l+2n))]i_{l+2}(\sqrt{\tau})+\tau\sqrt{\tau}(1-\sigma)l(l-1)i_{l+3}(\sqrt{\tau})\Big{\}}\\
=&\left(2l+\sigma n+1-\sigma)i_{l+1}(\sqrt{\tau})+\sqrt{\tau} i_{l+2}(\sqrt{\tau})\right)^{-1}\Big{\{}[\tau l(\sigma n-\sigma+2l+1+\sigma l-\sigma-l+1)\\
&+(1-\sigma)l(l-1)((l+n-2)(\sigma n-\sigma+\sigma l+1)+3l^2+2nl-2l)]i_{l+1}(\sqrt{\tau})\\
&+\sqrt{\tau} l(1-\sigma)(l-1)(4l+2n)i_{l+2}(\sqrt{\tau})+\tau\sqrt{\tau} li_{l+2}(\sqrt{\tau})+\tau(1-\sigma)l(l-1)i_{l+3}(\sqrt{\tau})\Big{\}}\\
=&\tau l+\left((2l+\sigma n+1-\sigma)i_{l+1}(\sqrt{\tau})+\sqrt{\tau} i_{l+2}(\sqrt{\tau})\right)^{-1}\Big{\{}[\tau l(\sigma l-\sigma-l+1)\\
&+(1-\sigma)l(l-1)((l+n-2)(\sigma n-\sigma+\sigma l+1)+3l^2+2nl-2l)]i_{l+1}(\sqrt{\tau})\\
&+\sqrt{\tau} l(1-\sigma)(l-1)(4l+2n)i_{l+2}(\sqrt{\tau})+\tau(1-\sigma)l(l-1)i_{l+3}(\sqrt{\tau})\Big{\}}\\
=&\tau l+\left((2l+\sigma n+1-\sigma)i_{l+1}(\sqrt{\tau})+\sqrt{\tau} i_{l+2}(\sqrt{\tau})\right)^{-1}\Big{\{}[\tau l(\sigma l-\sigma-l+1)\\
&+(1-\sigma)l(l-1)((l+n-2)(\sigma n-\sigma+\sigma l+1)+3l^2+2nl-2l)]i_{l+1}(\sqrt{\tau})+\\
&\sqrt{\tau}(2l+n-2)(l-1)l(1-\sigma)i_{l+2}(\sqrt{\tau})\Big{\}}.
\end{split}
\end{eqnarray*}
Hence, for $l\geq 2$, one has
\begin{eqnarray*}
\begin{split}
\lambda_{(l)}-\lambda_{(1)}=&\tau l+\left((2l+\sigma n+1-\sigma)i_{l+1}(\sqrt{\tau})+\sqrt{\tau} i_{l+2}(\sqrt{\tau})\right)^{-1}\Big{\{}[\tau l(\sigma l-\sigma-l+1)\\
&+(1-\sigma)l(l-1)((l+n-2)(\sigma n-\sigma+\sigma l+1)+3l^2+2nl-2l)]i_{l+1}(\sqrt{\tau})+\\
&\sqrt{\tau}(2l+n-2)(l-1)l(1-\sigma)i_{l+2}(\sqrt{\tau})\Big{\}}-\tau\\
&=\tau(l-1)+\left((2l+\sigma n+1-\sigma)i_{l+1}(\sqrt{\tau})+\sqrt{\tau} i_{l+2}(\sqrt{\tau})\right)^{-1}\Big{\{}[\tau l(\sigma l-\sigma-l+1)\\
&+(1-\sigma)l(l-1)((l+n-2)(\sigma n-\sigma+\sigma l+1)+3l^2+2nl-2l)]i_{l+1}(\sqrt{\tau})+\\
&\sqrt{\tau}(2l+n-2)(l-1)l(1-\sigma)i_{l+2}(\sqrt{\tau})\Big{\}}.
\end{split}
\end{eqnarray*}
Since $\sigma\in(-\frac{1}{n-1},1)$, it is easy to have
$\lambda_{(l)}-\lambda_{(1)}>0$, which implies
 \begin{eqnarray*}
\lambda_{(l)}>\lambda_{(1)}=\tau>0.
 \end{eqnarray*}
Therefore, we obtain that the first nonzero eigenvalue $\lambda_{2}$
of the problem (\ref{1-4}) is $\lambda_{(1)}$. For each
$l\in\mathbb{N}$, when $\Omega=\mathbf{B}_{1}$ is the unit ball in
$\mathbb{R}^{n}$ centered at the origin, one knows
\begin{eqnarray} \label{add-x6}
\lambda_{(l)}=\inf\frac{\int_{\mathbf{B}_{1}}\left[(1-\sigma)|D^2u|^2+\sigma|\Delta
u|^2+\tau|\nabla
u|^2\right]dx}{\int_{\partial\mathbf{B}_{1}}u^{2}dS},
\end{eqnarray}
where the infimum is taken among all functions $u$ that are
$L^2(\partial B)$-orthogonal to the first $m-1$ eigenfunctions
$u_i$, with $m\in\mathbb{ N}$, such that $\lambda_{(l)}=\lambda_m$
is the $m$-th eigenvalue of the eigenvalue problem (\ref{1-4}).
Then, after sorting $\lambda_{(l)}$, it is not hard to obtain a
non-negative spectrum of eigenvalues of the problem (\ref{1-4}). As
 we all known, eigenfunctions of the form
 $u_{l}=R_l(r)Y_{l}(\theta)$ should
achieve the infimum in (\ref{add-x6}). This completes the proof of
Theorem \ref{theorem3-2}.
\end{proof}

\section{The isoperimetric inequality}
\renewcommand{\thesection}{\arabic{section}}
\renewcommand{\theequation}{\thesection.\arabic{equation}}
\setcounter{equation}{0}

The so-called \emph{Fraenkel asymmetry} is the following: for any
open set $\Omega\in\mathbb{R}^n$ with finite measure,
 \begin{eqnarray*}
\mathcal{A}(\Omega):=\inf\left\{\frac{\|\chi_{\Omega}-\chi_{\mathbf{B}}\|_{L^1(\mathbb{R}^n)}}{|\Omega|}\Bigg{|}\mathbf{B}~\mathrm{
is ~the~ ball~ with~ } |B|=|\Omega|\right\},
 \end{eqnarray*}
  where $\mathcal{A}(\Omega)$ is the
distance in the $L^1(\mathbb{R}^n)$ norm of a set $\Omega$ from the
set of all balls of the same measure as $\Omega$. This quantity
turns out to be a suitable distance between sets for the purposes of
stability estimates of eigenvalues.

 In order to prove Theorem \ref{maintheorem}, we need the following
 two facts:

\begin{lemma} \label{lemma4-1}
(\cite{bpr}) Let $\Omega$ be an open set with Lipschitz boundary and $p>1$. Then
 \begin{eqnarray*}
\int_{\partial\Omega}|x|^pdS\geq\int_{\partial\Omega^{\ast}}|x|^pdS\left(1+c_{n,p}\left(\frac{|\Omega\Delta\Omega^{\ast}|}{|\Omega|}^2\right)\right),
 \end{eqnarray*}
where $\Omega^{\ast}$ is the ball centered at zero with the same
measure as $\Omega$, $\Omega\Delta\Omega^{\ast}$ is the symmetric
difference of $\Omega$ and $\Omega^{\ast}$, and $c_{n,p}$ is a
constant depending only on $n$ and $p$ given by
 \begin{eqnarray*}
c_{n,p}:=\frac{(n+p-1)(p-1)}{4}\frac{\sqrt[n]{2}-1}{n}\left(\mathop{\min}\limits_{t\in[1,\sqrt[n]{2}]}t^{p-1}\right).
 \end{eqnarray*}
\end{lemma}

Using a similar argument to that of \cite[Theorem 1]{hx}, we can get
the following result.

\begin{lemma}  \label{lemma4-2}
Let $\Omega$ be a bounded domain of class $C^1$ in $\mathbb{R}^n$.
Then the eigenvalues of problem (\ref{1-4}) on $\Omega$ satisfy
\begin{eqnarray}
\begin{split}
\sum_{l=k+1}^{k+n}\frac{1}{\lambda_l(\Omega)}=\max\left\{\sum_{l=k+1}^{k+n}\int_{\partial\Omega}v_{l}^{2}dS\right\}.
\end{split}
\end{eqnarray}
Moreover, if the families $\{v_{l}\}^{l=k+n}_{l=k+1}$ satisfy
 \begin{eqnarray*}
 \int_{\Omega}\left[(1-\sigma)D^2v_i : D^2v_j +\sigma\Delta v_i\cdot\Delta v_j+\tau\nabla v_i\cdot\nabla v_j\right]dx=\delta_{ij}
 \end{eqnarray*}
 and $\int_{\partial\Omega} v_i u_jdS=0$
 for all $i=k+1,\ldots,k+n$ in $H^2(\Omega)$,
 where $u_1,u_2,\ldots,u_k$ are the first $k$ eigenfunctions of the problem (\ref{1-4}), then the maximum value can be attained.
\end{lemma}

Now, we can prove:

\begin{theorem}  \label{theorem4-3}
For every domain $\Omega$ in $\mathbb{R}^n$ of class $C^1$, we have
\begin{eqnarray} \label{add-4-2}
\begin{split}
\lambda_{2}(\Omega)\leq\lambda_{2}(\Omega^*)\left(1-\delta_n\mathcal{A}(\Omega)^{2}\right),
\end{split}
\end{eqnarray}
where $\delta_n$ is given by
$$\delta_n:=\frac{n+1}{8n}(\sqrt[n]{2}-1),$$
and $\Omega^*$ is a ball with the same measure as $\Omega$.
\end{theorem}

\begin{proof}
We divide our proof into two steps:

\textbf{Step 1}. Assume first that $\Omega$ is a bounded domain of
class $C^1$ in $\mathbb{R}^n$ with the same measure as the unit
Euclidean ball $\mathbf{B}_{1}$. Let trial functions be defined by
$v_l=(\tau|\Omega|)^{-\frac{1}{2}}x_l$, with $l=2,\ldots,n+1$, such
that they have zero integral mean over $\partial\Omega$. This can be
always assured. In fact, by coordinate transformation
$x=y-\frac{1}{|\partial\Omega|}\int_{\partial\Omega}ydS$, it is easy
to know that the trial functions have zero integral mean over
$\partial\Omega$. Moreover, functions $v_l$ also satisfy the
normalization condition of Lemma \ref{lemma4-2}. Choosing $k=1$,
$l=2$, $p=2$, and applying Lemma \ref{lemma4-1}, we can obtain
\begin{eqnarray} \label{add-4-3}
\begin{split}
\sum_{l=2}^{n+1}\frac{1}{\lambda_l(\Omega)}&\geq\frac{1}{\tau|\Omega|}\int_{\partial\Omega}|x|^2dS\nonumber\\
&\geq\frac{1}{\tau|\Omega|}\left(1+c_{n,2}\left(\frac{|\Omega\Delta B|}{|\Omega|}\right)^2\right)\int_{\partial B}|x|^2dS\nonumber\\
&\geq\sum_{l=2}^{n+1}\frac{1}{\lambda_l(B)}\left(1+c_{n,2}\left(\frac{|\Omega\Delta
B|}{|\Omega|}\right)^2\right).
\end{split}
\end{eqnarray}
Recall that $\lambda_{2}(\Omega)\leq\lambda_{l}(\Omega)$ for $l\geq
3$. Then, combing (\ref{add-4-3}) and the definition of
$\mathcal{A}(\Omega)$, one has
 \begin{eqnarray*}
\lambda_{2}(\Omega)(1+c_{n,2}\mathcal{A}(\Omega)^2)\leq\lambda_{2}(\mathbf{B}_{1}).
 \end{eqnarray*}
Clearly, the above inequality implies the eigenvalue inequality with
$\delta_n=\frac{1}{8}\min\{1,\frac{n+1}{n}(\sqrt[n]{2}-1)\}=\frac{n+1}{8n}(\sqrt[n]{2}-1)$.
So far, this step deals with the case that
$|\Omega|=|\mathbf{B}_{1}|$ completely and successfully.

\textbf{Step 2}. This step deals with the proof for general finite
values of $|\Omega|$, which relies on the scaling properties of the
eigenvalues.  For all $s>0$, noting
\begin{eqnarray*}
s\Omega:=\left\{x\in\mathbb{R}^{n}\Bigg{|}\frac{x}{s}\in\Omega\right\}.
 \end{eqnarray*}
Let $\lambda(\tau,\sigma,\Omega)$ be an eigenvalue of the eigenvalue
problem (\ref{1-4}). For any $u\in H^2(\Omega)$ with $\int_\Omega
udx=0$, let $ \widetilde{u}(x)$ be a valid trial function on
$s\Omega$ that makes $ \widetilde{u}(x) = u(\frac{x}{s})$, and then
the Rayleigh quotient
$\mathcal{Q}_{s^{-2}\tau,\sigma,s\Omega}[\widetilde{u}]$ can be
treated as follows:
\begin{eqnarray*}
\begin{split}
\mathcal{Q}_{s^{-2}\tau,\sigma,s\Omega}[\widetilde{u}]=&\frac{\int_{s\Omega}\left[(1-\sigma)|D^2\widetilde{u}|^2+\sigma|\Delta\widetilde{u}|^2+s^{-2}\tau|\nabla\widetilde{u}|^2\right]dx}
{\int_{\partial s\Omega}\widetilde{u}^2dS}\\
=&\frac{\int_{s\Omega}\left[(1-\sigma)|s^{-2}D^2u(\frac{x}{s})|^2+\sigma|s^{-2}\Delta
u(\frac{x}{u})|^2+s^{-2}\tau|s^{-1}\nabla
u(\frac{x}{u})|^2\right]dx}
{\int_{\partial s\Omega}u(\frac{x}{s})^2dS}\\
(\mathrm{taking}~ y=\frac{x}{s}) \
=&s^{-4+n}\frac{\int_{\Omega}\left[(1-\sigma)|D^2u|^2+\sigma|\Delta
u|^2+\tau|\nabla u|^{2}\right]dy}
{s^{n-1}{\int_{\partial\Omega}u^2d\sigma}}\\
=&s^{-3}\mathcal{Q}_{\tau,\sigma,\Omega}[u],
\end{split}
\end{eqnarray*}
which implies
\begin{eqnarray*}
\lambda(\tau,\sigma,\Omega)=s^3\lambda(s^{-2}\tau,\sigma,s\Omega).
\end{eqnarray*}
Together this scaling property with the conclusion in \textbf{Step
1}, we know that the eigenvalue inequality (\ref{add-4-2}) also
holds for general bounded domains of class $C^1$ in $\mathbb{R}^n$.
This completes the proof of Theorem \ref{theorem4-3}.
\end{proof}

\begin{proof} [Proof of Theorem \ref{maintheorem}]
The conclusion of Theorem \ref{maintheorem} follows directly by
applying Theorem \ref{theorem4-3} and the fact that
$\mathcal{A}(\Omega^{\ast})=0$.
\end{proof}

\section{Further study}
\renewcommand{\thesection}{\arabic{section}}
\renewcommand{\theequation}{\thesection.\arabic{equation}}
\setcounter{equation}{0}

Theorem \ref{maintheorem} tells us that for the functional
$\lambda_{2}(\cdot):\Omega\mapsto\mathbb{R}^{+}$ with $|\Omega|$
fixed, $\Omega^{\ast}$ is a maximum value of $\lambda_{2}(\cdot)$. A
natural question is:

$\\$\textbf{Open problem}. \emph{Is $\Omega^{\ast}$ the only maximum
value of the functional $\lambda_{2}(\cdot)$ defined above? Can we
get the rigidity from the isoperimetric inequality
$\lambda_{2}(\Omega)\leq\lambda_{2}(\Omega^{\ast})$ with $|\Omega|$
fixed? Is the claim ``if the equality in Theorem \ref{maintheorem}
holds, then $\Omega$ must be $\Omega^{\ast}$." true?}

$\\$ We believe that the answer to the above \textbf{Open problem}
is positive. However, so far, we do not know how to prove it. We
have tried the approach shown in \cite[Subsection 4.1]{bp} to
explain that $\Omega^{\ast}$ would be the unique critical point of
the functional $\lambda_{2}(\cdot)$, but we fail.
 Clearly, we definitely hope that this difficulty can be
overcome in the future.

\section*{Acknowledgments}
\renewcommand{\thesection}{\arabic{section}}
\renewcommand{\theequation}{\thesection.\arabic{equation}}
\setcounter{equation}{0} \setcounter{maintheorem}{0}

This work is partially supported by the NSF of China (Grant Nos.
11801496 and 11926352), the Fok Ying-Tung Education Foundation
(China) and  Hubei Key Laboratory of Applied Mathematics (Hubei
University).


\begin{thebibliography}{9999}



\bibitem{as} M. Abramowitz, I. A. Stegun, \emph{Handbook of Mathematical Functions with Formulas, Graphs, and Mathematical
Tables}, Natl. Bur. Stand., Appl. Math. Ser., U.S. Government
Printing Office, Washington, DC, 1964.


\bibitem{ab} M.-S. Ashbaugh, R.-D. Benguria, \emph{On Rayleigh's conjecture for the clamped plate and its generalization to three
dimensions}, Duke Math. J. {\bf78}(1) (1995) 1--7.


\bibitem{bpr} L. Brasco, G. De Philippis, B. Ruffini, \emph{Spectral optimization for
the Stekloff-Laplacian: the stability issue}, J. Funct. Anal. {\bf
262}(11) (2012) 4675--4710.


\bibitem{bro} F. Brock, \emph{An isoperimetric inequality for eigenvalues of the
Stekloff problem}, Z. Angew. Math. Mech. {\bf81}(1) (2001) 69--71.


\bibitem{bp} D. Buoso, L. Provenzano, \emph{A few shape optimization results for a biharmonic
Steklov problem}, J. Differ. Equat. {\bf259}(5) (2015) 1778--1818.

\bibitem{lmc1} L.-M. Chasman, \emph{An isoperimetric inequality for fundamental tones
of free plates}, Commun. Math. Phys. {\bf 303} (2011) 421--449.

\bibitem{lmc2} L.-M. Chasman, \emph{An isoperimetric inequality for fundamental tones of free plates with
nonzero Poisson's ratio}, Applicable Analysis {\bf 95} (2016)
1700--1735.


\bibitem{dmwxy} F. Du, J. Mao, Q.-L. Wang, C.-Y. Xia, Y. Zhao,
\emph{Estimates for eigenvalues of the Neumann and Steklov
problems}, avilable online at arXiv:1902.08998.


\bibitem{fa} G. Faber, \emph{Beweis, da{\ss} unter allen homogenen Membranen von gleicher Fl\"{a}che und gleicher Spannung die
kreisf\"{o}rmige den tiefsten Grundton gibt}, Sitz. Ber. Bayer.
Akad. Wiss. (1923) 169--172.

\bibitem{kr} E. Krahn, \emph{\"{U}ber eine von Rayleigh formulierte Minimaleigenschaft des Kreises}, Math. Ann. {\bf94} (1924)
97--100.

\bibitem{ggs} F. Gazzola, H. C. Grunau, G. Sweers, \emph{Polyharmonic Boundary Value Problems. Positivity Preserving and Nonlinear
Higher Order Elliptic Equations in Bounded Domains}, Lecture Notes
in Math., Springer-Verlag, Berlin, 2010.


\bibitem{hx} G. N. Hile, Z. Y. Xu, \emph{Inequalities for sums of reciprocals of
eigenvalues}, J. Math. Anal. Appl. {\bf180}(2) (1993) 412--430.


\bibitem{lp}  P.-D. Lamberti, L. Provenzano, Viewing the Steklov eigenvalues of the Laplace operator as critical Neumann
eigenvalues, in: Current Trends in Analysis and Its Applications:
Proceedings of the 9th ISAAC Congress, Birkh\"{a}user, Krak\'{o}w,
2013, pp. 171--178, 2015, Basel.


\bibitem{ms} S. Li, J. Mao, \emph{Estimates for sums of eigenvalues of the free plate with
nonzero Poisson's ratio}, Proc. Amer. Math. Soc. 149(5) (2021)
2167--2177.

\bibitem{nsn} N.-S. Nadirashvili, \emph{Rayleigh's conjecture on the principal frequency
of the clamped plate}, Arch. Ration. Mech. Anal. {\bf129}(1) (1995)
1--10.


\bibitem{gs1}  G. Szeg\H{o}, \emph{On membranes and plates}, Proc. Nat. Acad. Sci. {\bf 36} (1950) 210--216.

\bibitem{gs2}  G. Szeg\H{o}, \emph{Note to my paper ``On membranes and
plates"}, Proc. Nat. Acad. Sci. {\bf 44} (1958) 314--316.

\bibitem{hfw} H.-F. Weinberg, \emph{ An isoperimetric inequality for the $N$-dimensional free membrane
problem},  J. Rational Mech. Anal. {\bf 5}  (1956) 633--636.


\end{thebibliography}
\end{document}